\documentclass{svmult}

\usepackage{latexsym,amssymb,amsgen,amsmath,amsxtra,amsfonts,mathtools}
\usepackage{mathrsfs}
\usepackage[mathscr]{euscript}

\usepackage{tikzexternal} 
\usepackage{microtype}
\usepackage{url}
\usepackage{booktabs}

\DeclareMathOperator{\sEnd}{\mathscr{E}\hspace{-1pt}\mathit{nd}}

\smartqed 

\begin{document}

\motto[0.5\textwidth]{The correspondence between Higgs bundles and local systems can be viewed as a Hodge theorem for non-abelian cohomology. \\ \flushright{--- Carlos T. Simpson \cite{MR1179076}}
}

\title*{Introduction to \\ Nonabelian Hodge Theory}
\subtitle{Flat connections, Higgs bundles and \\ complex variations of Hodge structure}
\author{Alberto Garc{\'\i}a-Raboso \and Steven Rayan}
\institute{Alberto Garc{\'\i}a-Raboso,
  \email{agraboso@math.toronto.edu}\\
  Steven Rayan,
  \email{stever@math.toronto.edu}
}

\maketitle

\setcounter{minitocdepth}{2}
\dominitoc

\abstract*{}

\section{Introduction}

Hodge theory bridges the topological, smooth and holomorphic worlds. In the abelian case, these are embodied by the Betti, de Rham and Dolbeault cohomology groups, respectively, of a smooth compact K{\"a}hler manifold, $X$. The standard isomorphisms
$$
H^1_B(X, \mathbb{C}) \cong \operatorname{Hom}(H_1(X, \mathbb{Z}), \mathbb{C}) \cong \operatorname{Hom}(\pi_1 X, \mathbb{C})
$$
then assert that we are, in particular, looking at representations of the fundamental group of $X$ in the additive group of the complex numbers.

Therefore, it is only natural to ask which smooth and holomorphic objects realize representations of $\pi_1 X$ in other, possibly nonabelian, groups. In the first case, a complete answer is given by the Riemann-Hilbert correspondence: representations of $\pi_1 X$ in the general linear group ${\fontfamily{cmss}\selectfont\textbf{GL}}_r(\mathbb{C})$ (up to conjugacy) are equivalent to $C^\infty$ vector bundles on $X$ equipped with a flat connection.

The first results making a connection with the holomorphic world appeared in the works of Narasimhan and Seshadri \cite{MR0184252,MR0233371} in the 1960s: they discovered that unitary representations of the fundamental group of a compact Riemann surface yield stable\footnote{The appropriate stability condition in this setting is \emph{slope stability}; for holomorphic vector bundles, this is the same as Definition \ref{definition:Stability} with $\phi = 0$.} holomorphic vector bundles. Their correspondence was extended to higher-dimensional compact K{\"a}hler manifolds by Donaldson \cite{MR765366,MR885784}, Mehta and Ramanathan \cite{MR751136}, and Uhlenbeck and Yau \cite{MR861491}.

A seminal paper of Hitchin's, \cite{MR887284}, introduced the notion of a Higgs field, the data of which incorporates the non-unitary part of a ${\fontfamily{cmss}\selectfont\textbf{GL}}_r(\mathbb{C})$-representation. A flurry of activity ---Donaldson \cite{MR887285}, Diederich and Ohsawa \cite{MR817167}, Beilinson and Deligne (unpublished), Corlette \cite{MR965220}, Jost and Yau \cite{MR690192,MR1173045}, Simpson \cite{MR944577}--- culminated in Simpson's correspondence \cite{MR1179076} between ${\fontfamily{cmss}\selectfont\textbf{GL}}_r(\mathbb{C})$-representations of $\pi_1 X$ and holomorphic vector bundles on $X$ equipped with a Higgs field.

It is this \emph{nonabelian Hodge theorem} that the first part of these notes concentrates on. After describing the main actors in \S\ref{section:ZoologyOfConnections}, we will state it in \S\ref{section:CorrespondenceHiggsBundlesAndLocalSystems} and give an outline of its proof (closely following \cite{MR1179076}) in \S\ref{section:ASketchOfProof}. In Figure \ref{figure:Diagram} we have represented all of the objects that will appear in \S\ref{section:ZoologyOfConnections}--\S\ref{section:ASketchOfProof} with an emphasis on the world in which each one of them lives, and how they are described (what we call contexts; cf. \S\ref{section:Differentials}). We hope it will help the reader succesfully navigate these sections.

Although we have tried to make our discussion here as self-contained as possible, those wishing for extended expositions on these and other related concepts can refer to the textbooks \cite{MR909698}, \cite{MR2093043} and \cite{MR2359489}.

In the second part (\S\ref{section:TheHyperkaehlerManifold}--\S\ref{section:ExampleCalculations}), we move to the level of moduli spaces: those of Higgs bundles, flat connections, and representations of the fundamental group.  Simpson calls these the Dolbeault, de Rham, and Betti moduli spaces, respectively \cite{MR1307297}.  We focus in particular on the moduli space of Higgs bundles on a curve, because another aspect of Hodge theory ---complex variations of Hodge structure--- arises naturally from the $\mathbb{C}^\times$-action on this space.  (In this survey, we concentrate on the case of ordinary Higgs bundles; for more general Hitchin systems, in which the Higgs field is permitted to take values in an arbitrary line bundle, see \cite{MR1085642,MR1334607,MR1300764}. The genus zero case is treated in, e.g., \cite{MR1049157,MR1086744}.)  In section \S\ref{section:CircleActionAndComplexVariationsOfHodgeStructure}, we explain how the $\mathbb{C}^\times$-action enables one to calculate Betti numbers of the moduli space of Higgs bundles.  As a consequence of the nonabelian Hodge theorem, these Betti numbers are also those of the de Rham and Betti spaces.

For a concrete example of the role that complex variations of Hodge structure play in the topology of Higgs bundle moduli spaces, we give in \S\ref{section:ExampleCalculations} an outline of Hitchin's calculation of the Betti numbers for the space of rank $2$ Higgs bundles.  Historically, after Hitchin's calculation in \cite{MR887284}, the Betti numbers for rank $3$ were computed by Gothen \cite{MR1298999,PBG:95} and the Betti numbers for rank $4$ were calculated by Garc{\'\i}a-Prada, Heinloth, and Schmitt \cite{GHS:11}.  All of these calculations stem from Hitchin's Morse-theoretic localization technique \cite{MR887284}, although the rank $4$ calculation is motivic in nature.

Computing the Betti numbers of these noncompact spaces is a difficult problem, even with the $\mathbb{C}^\times$-action.  There are conjectural Betti numbers for Higgs bundle moduli spaces of all ranks, derived arithmetically by Hausel and Rodriguez-Villegas \cite{MR2453601}, and using physical arguments by Chuang, Diaconescu, and Pan \cite{MR2833316}, demonstrating the diversity of tools employed in the study of Higgs bundles.

In these notes, we are admittedly only skimming a bit of cream from the top of a rich theory.  Dig deeper and you will find:
\begin{itemize}
\item the nonabelian Hodge correspondence for groups other than $\fontfamily{cmss}\selectfont\textbf{GL}_r$ \cite{MR1179076};
\item the nonabelian Hodge correspondence over quasiprojective varieties, in its tame \cite{MR575939,MR1040197,MR1422313,MR1617644,MR2310103,MR2469521} and wild \cite{MR2004129,arXiv1203.6607B,MR2919903,MR2459120} variants;
\item moduli spaces of parabolic Higgs bundles, e.g., \cite{MR1206652,MR1375314,MR1411302,MR1397988,MR2308696,MR2746468};
\item connections to generalized complex geometry, e.g., \cite{MR2747007};
\item the central role of Higgs bundles in the geometric Langlands conjecture, e.g., \cite{MR2537083,MR2648685};
\item the pivotal part played by Higgs bundles in Ng\^o's proof of the fundamental lemma \cite{MR2653248}.
\end{itemize}
Of course, any errors, omissions, or misrepresentations in this survey are our fault alone.

\medskip
\emph{Acknowledgements.} We thank Chuck Doran, David Morrison, Radu Laza, and Johannes Walcher for organizing the Workshop on Hodge Theory in String Theory, as well as Alan Thompson for arranging the concentrated graduate course in which we gave these lectures.  We are indebted to Noriko Yui for encouraging us to contribute these notes, and for organizing a vibrant and productive thematic program at the Fields Institute.  These notes benefited from useful comments and suggestions by Philip Boalch, Tam{\'a}s Hausel, Tony Pantev, and an anonymous referee.  We also thank Marco Gualtieri and Lisa Jeffrey for their support and for providing a stimulating working environment in Toronto.

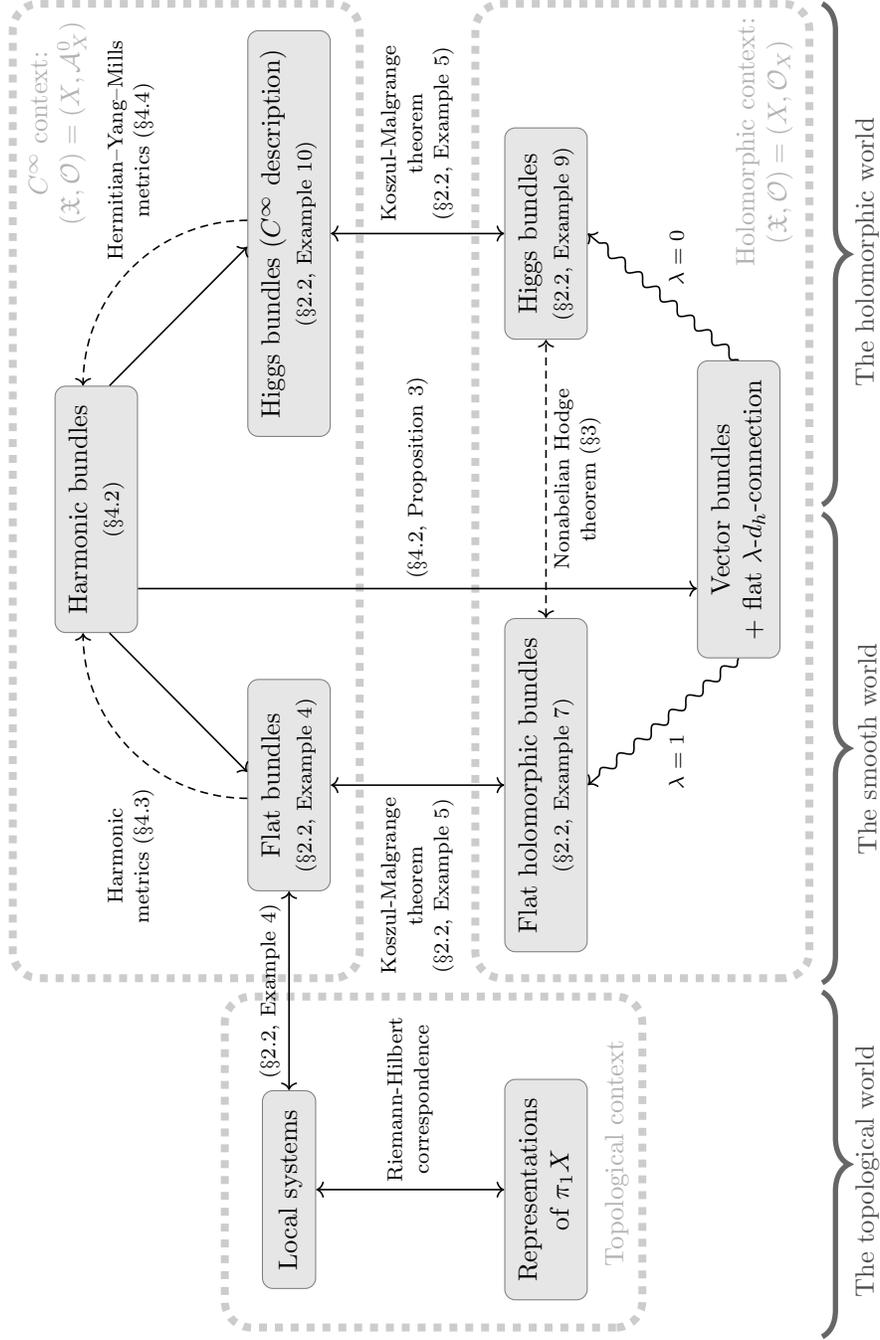
\begin{figure}[p]
\centering
\rotatebox{90}{
\begin{tikzpicture}[x=348pt,y=243pt,
  boxed/.style={rectangle, 
                rounded corners=4pt,
                draw=black!50,
                fill=black!10,
                align=center,
                font=\normalsize,
                text height=2.5em,
                text depth=.5em,
                inner xsep=7pt}]  
  \node[boxed] (Reps) at (-0.04,0.3)
    {Representations \\ of $\pi_1 X$};
  \node[boxed,text height=1.2em] (LocalSystems) at (-0.04,0.7)
    {Local systems};
  \node[boxed] (FlatBundles) at (0.4,0.7)
    {Flat bundles \\ 
      \footnotesize{(\S\ref{section:Connections},
      Example \ref{example:FlatConnections})}
    };
  \node[boxed] (FlatHoloBundles) at (0.4,0.3)
    {Flat holomorphic bundles \\
      \footnotesize{(\S\ref{section:Connections},
      Example \ref{example:HolomorphicFlatConnections})}
    };
  \node[boxed] (HarmonicBundles) at (0.7,1)
    {Harmonic bundles \\
      \footnotesize{(\S\ref{section:HarmonicBundles})}
    };
  \node[boxed] (HiggsBundlesCInfty) at (1,0.7)
    {Higgs bundles ($C^\infty$ description) \\
      \footnotesize{(\S\ref{section:Connections},
      Example \ref{example:HiggsFieldsSmoothDescription})}
    };
  \node[boxed] (HiggsBundles) at (1,0.3)
    {Higgs bundles \\
      \footnotesize{(\S\ref{section:Connections},
      Example \ref{example:HiggsFields})}
    };
  \node[boxed] (HoloLambdaConnections) at (0.7,0)
    {Vector bundles \\
      $+$ flat $\lambda$-$d_h$-connection
    };
  \path[semithick,<->]
    (LocalSystems) edge node[right,align=center]
    {Riemann-Hilbert \\ correspondence} (Reps);
  \path[semithick,<->]
    (LocalSystems) edge node[above,align=center]
    {(\S\ref{section:Connections}, Example \ref{example:FlatConnections})}
    (FlatBundles);
  \path[semithick,<->]
    (HiggsBundlesCInfty) edge node[xshift=38pt,align=center]
    {Koszul-Malgrange \\ theorem \\ (\S\ref{section:Connections},
    Example \ref{example:HolomorphicStructures})} (HiggsBundles);
  \path[semithick,<->]
    (FlatBundles) edge node[xshift=-38pt,align=center]
    {Koszul-Malgrange \\ theorem \\ (\S\ref{section:Connections},
    Example \ref{example:HolomorphicStructures})} (FlatHoloBundles);
  \path[semithick,->]
    ([xshift=-30pt]HarmonicBundles.south) edge node[xshift=42pt,align=center]
    {(\S\ref{section:HarmonicBundles}, Proposition
    \ref{proposition:LambdaConnectionsFromHarmonicBundles})}
    ([xshift=-30pt]HoloLambdaConnections.north);
  \path[semithick,densely dashed,<->]
    (FlatHoloBundles) edge node[below,align=center]
    {Nonabelian Hodge \\ theorem
    (\S\ref{section:CorrespondenceHiggsBundlesAndLocalSystems})}
    (HiggsBundles);
  \path[semithick,->]
    ([yshift=-5pt]HarmonicBundles.west) edge
    ([xshift=5pt]FlatBundles.north);
  \path[semithick,->]
    ([yshift=-5pt]HarmonicBundles.east) edge
    ([xshift=-5pt]HiggsBundlesCInfty.north);
  \path[semithick,densely dashed,->]
    ([xshift=-5pt]FlatBundles.north) edge[out=90,in=180]
    node[ xshift=-40pt,align=center] {Harmonic \\ metrics
    (\S\ref{section:FromFlatBundlesToHarmonicBundles})}
    ([yshift=5pt]HarmonicBundles.west);
  \path[semithick,densely dashed,->]
    ([xshift=5pt]HiggsBundlesCInfty.north) edge[out=90,in=0]
    node[ xshift=45pt,align=center] {Hermitian--Yang--Mills \\
    metrics (\S\ref{section:FromHiggsBundlesToHarmonicBundles})}
    ([yshift=5pt]HarmonicBundles.east);
  \path[semithick,->]
    (HoloLambdaConnections.east) edge[decorate,decoration={snake,
    amplitude=.6mm,segment length=3mm,post length=.7mm}]
    node[xshift=15pt,yshift=-5pt] {$\lambda=0$} (HiggsBundles.south);
  \path[semithick,->]
    (HoloLambdaConnections.west) edge[decorate,decoration={snake,
    amplitude=.6mm,segment length=3mm,post length=.7mm}]
    node[xshift=-15pt,yshift=-5pt] {$\lambda=1$} (FlatHoloBundles.south);
  \begin{pgfonlayer}{background}
    \draw[black!60,line width=2pt,decorate,
      decoration={brace,amplitude=10pt,mirror}]
      (-0.2,-0.13) -- (0.175,-0.13);
    \node[black!70,font=\normalsize] at (-0.02,-0.2)
      {The topological world};
    \draw[black!60,line width=2pt,decorate,
      decoration={brace,amplitude=10pt,mirror}]
      (0.185,-0.13) -- (0.695,-0.13);
    \node[black!70,font=\normalsize] at (0.44,-0.2)
      {The smooth world};
    \draw[black!60,line width=2pt,decorate,
      decoration={brace,amplitude=10pt,mirror}]
      (0.705,-0.13) -- (1.25,-0.13);
    \node[black!70,font=\normalsize] at (0.98,-0.2)
      {The holomorphic world};
    \draw[black!20,line width=3pt,dashed,rounded corners=10pt]
      (-0.19,0.15) rectangle (0.17,0.8);
    \node[black!30,font=\normalsize] at (0,0.19)
      {Topological context};
    \draw[black!20,line width=3pt,dashed,rounded corners=10pt]
      (0.19,0.59) rectangle (1.25,1.13);    
    \node[black!30,font=\normalsize,align=center] at (1.12,1.06)
      {$C^\infty$ context: \\
        $(\mathfrak{X}, \mathcal{O}) = (X, \mathcal{A}^0_X)$
      };
    \draw[black!20,line width=3pt,dashed,rounded corners=10pt]
      (0.19,-0.11) rectangle (1.25,0.41);
    \node[black!30,font=\normalsize,align=center] at (1.1,-0.04)
      {Holomorphic context: \\
        $(\mathfrak{X}, \mathcal{O}) = (X, \mathcal{O}_X)$
      };
  \end{pgfonlayer}
\end{tikzpicture}
}
\caption{The objects in nonabelian Hodge theory and their relationships.}
\label{figure:Diagram}
\end{figure}

\section{Zoology of connections}
\label{section:ZoologyOfConnections}

Both the statement and the inner workings of the nonabelian Hodge correspondence revolve around a collection of operators that are generically known as \emph{connections}. In order to accommodate all the cases we will have need for, we offer some general definitions and statements immediately followed by a detailed discussion of those cases.

\subsection{Differentials}
\label{section:Differentials}

Let $\mathfrak{X}$ be a topological space, $\mathcal{O}$ a sheaf of commutative $\mathbb{C}$-algebras on $\mathfrak{X}$ making the pair $(\mathfrak{X}, \mathcal{O})$ into a locally-ringed space, and $\mathcal{K}$ a locally-free sheaf of coherent $\mathcal{O}$-modules on $\mathfrak{X}$\footnote{The reader unfamiliar with the language of locally-ringed spaces should think of the pair $(\mathfrak{X}, \mathcal{O})$ as selecting a geometric context ---the examples below should make this clear. ``A locally-free sheaf of coherent $\mathcal{O}$-modules on $\mathfrak{X}$'' is then just a way of saying ``a vector bundle of finite rank'' in that context.}.

\begin{definition} \label{def:Differential}
A \emph{differential on the pair $(\mathcal{O}, \mathcal{K})$} is a degree 1, $\mathbb{C}$-linear derivation of the exterior algebra $\bigwedge^\bullet_\mathcal{O} \mathcal{K}$
that squares to zero.
\end{definition}
In concrete terms, a differential consists of a collection of $\mathbb{C}$-linear maps 
$$
D^{(k)} : \textstyle{\bigwedge^k_\mathcal{O}} \mathcal{K} \to \textstyle{\bigwedge^{k+1}_\mathcal{O}} \mathcal{K}, \qquad k\geq 0
$$
satisfying the graded Leibniz rule
$$
D^{(k_1+k_2)} (\omega_1 \wedge \omega_2) = D^{(k_1)}(\omega_1) \wedge \omega_2 + (-1)^{k_1} \omega_1 \wedge D^{(k_2)}\omega_2
$$
for $\omega_1 \in \bigwedge^{k_1}_\mathcal{O} \mathcal{K}$ and $\omega_2 \in \bigwedge^{k_2}_\mathcal{O} \mathcal{K}$, and such that $D^{(k+1)} \circ D^{(k)} = 0$. We will follow the common practice of omitting the superscript from the notation, denoting simply by $D$ all of the above maps ---which leaves us with the tantalising equation $D^2=0$.

Observe that this data turns
$$
0 \longrightarrow \mathcal{O} \xrightarrow{\;D\;} \mathcal{K} \xrightarrow{\;D\;} \textstyle{\bigwedge^2_\mathcal{O}} \mathcal{K} \xrightarrow{\;D\;} \cdots
$$
into a cochain complex of sheaves of $\mathbb{C}$-vector spaces. In fact, the compatibility with products expressed by the graded Leibniz rule makes the pair $\left( \bigwedge^\bullet_\mathcal{O} \mathcal{K}, D\right)$ into a cdga (=commutative differential graded algebra) object in the category of sheaves of $\mathbb{C}$-vector spaces on $\mathfrak{X}$, and its hypercohomology, $\mathbb{H}^k \!\left( X, \left( \bigwedge^\bullet_\mathcal{O} \mathcal{K}, D\right) \right)$, into a graded commutative $\mathbb{C}$-algebra.

\clearpage
\begin{center}
\textbf{Examples}
\end{center}

\begin{example} \label{example:deRhamDifferential}
Let $\mathfrak{X} = X$ be a compact K{\"a}hler manifold, $\mathcal{O} = \mathcal{A}^0_X$ the sheaf of smooth complex-valued functions on it, and $\mathcal{K} = \mathcal{A}^1_X$ the sheaf of smooth complex-valued 1-forms on $X$. The exterior derivative of smooth complex-valued differential forms,
\begin{equation} \label{eq:deRhamDifferential}
d: \mathcal{A}^k_X \to \mathcal{A}^{k+1}_X, \qquad k\geq 0,
\end{equation}
is then a differential on the pair $(\mathcal{A}^0_X, \mathcal{A}^1_X)$. The associated complex,
\begin{equation} \label{eq:deRhamComplex}
0 \longrightarrow \mathcal{A}^0_X \xrightarrow{\;d\;} \mathcal{A}^1_X \xrightarrow{\;d\;} \mathcal{A}^2_X \xrightarrow{\;d\;} \cdots,
\end{equation}
is known as the \emph{de Rham complex} of $X$, for it computes the \emph{de Rham cohomology} of $X$:
\begin{equation} \label{eq:deRhamCohomology}
H_\mathrm{dR}^k(X) = H^k \big( \Gamma(X, \mathcal{A}^\bullet_X), d \big) \cong \mathbb{H}^k \big( X, (\mathcal{A}^\bullet_X, d) \big).
\end{equation}
\end{example}

\begin{example} \label{example:DolbeaultDifferentials}
Recall that, on a complex manifold, we can separate smooth complex-valued $k$-forms into their $(p,q)$ components:
$$
\mathcal{A}^k_X \cong \bigoplus_{p+q=k}\mathcal{A}^{p,q}_X.
$$
The image under the de Rham differential (cf. Example \ref{example:deRhamDifferential}) of a $k$-form $\omega$ of type $(p,q)$ splits as the sum of two terms,
$$
d\omega = \partial\omega + \overline\partial\omega,
$$
where $\partial \omega$ and $\overline\partial \omega$ are $(k+1)$-forms of type $(p+1,q)$ and $(p,q+1)$, respectively. These so-called \emph{Dolbeault operators}, $\partial$ and $\overline\partial$, restrict to differentials on the pairs $(\mathcal{A}^0_X, \mathcal{A}^{1,0}_X)$ and $(\mathcal{A}^0_X, \mathcal{A}^{0,1}_X)$, respectively. The cochain complex associated to the second of these is called the \emph{Dolbeault complex} of $X$,
$$
0 \longrightarrow \mathcal{A}^0_X \xrightarrow{\;\overline\partial\;} \mathcal{A}^{0,1}_X \xrightarrow{\;\overline\partial\;} \mathcal{A}^{0,2}_X \xrightarrow{\;\overline\partial\;} \cdots,
$$
and it gives rise to some of the \emph{Dolbeault cohomology groups} of $X$:
$$
H_{\overline\partial}^{0,k}(X) = H^k \big( \Gamma(X, \mathcal{A}^{0,\bullet}_X), \overline\partial \big) \cong \mathbb{H}^k \big( X, (\mathcal{A}^{0,\bullet}_X, \overline\partial) \big).
$$
The Dolbeault cohomology groups $H^{p,q}_{\overline\partial}(X)$ for $p \neq 0$ will appear later on in our discussion.
\end{example}

\begin{example} \label{example:HolomorphicdeRhamDifferential}
Let $\mathcal{O} = \mathcal{O}_X = \ker \{ \overline\partial : \mathcal{A}^0_X \to \mathcal{A}^{0,1}_X\}$ be the sheaf of holomorphic functions on $X$. The $\mathcal{A}^0_X$-module structure on $\mathcal{A}^{1,0}_X$ naturally restricts to an $\mathcal{O}_X$-module structure on the sheaf $\mathcal{K} = \Omega^1_X = \ker \{ \overline\partial : \mathcal{A}^{1,0}_X \to \mathcal{A}^{1,1}_X\}$ of holomorphic 1-forms on $X$. The \emph{holomorphic de Rham differential},
$$
d_h: \Omega^k_X \to \Omega^{k+1}_X, \qquad k\geq 0,
$$
induced by the de Rham differential on smooth complex-valued differential forms, then becomes a differential on the pair $(\mathcal{O}_X, \Omega^1_X)$. The associated complex,
$$
0 \longrightarrow \mathcal{O}_X \xrightarrow{\;d_h\;} \Omega^1_X \xrightarrow{\;d_h\;} \Omega^2_X \xrightarrow{\;d_h\;} \cdots,
$$
receives the name of \emph{holomorphic de Rham complex} of $X$, and it turns out to also compute the de Rham cohomology of $X$.
\end{example}

\subsection{Connections}
\label{section:Connections}

\begin{definition} \label{def:lambdaDConnection}
Let $D$ be a differential on $(\mathcal{O}, \mathcal{K})$, $\mathcal{V}$ a locally-free sheaf of coherent $\mathcal{O}$-modules on $\mathfrak{X}$, and $\lambda \in \mathbb{C}$. A \emph{$\lambda$-$D$-connection} on $\mathcal{V}$ is a $\mathbb{C}$-linear map, $\nabla: \mathcal{V} \to \mathcal{V} \otimes_\mathcal{O} \mathcal{K}$, satisfying the \emph{$\lambda$-twisted Leibniz rule},
$$
\nabla(fv) = \lambda v \otimes Df + f\nabla v, \qquad\text{for } f \in \mathcal{O}, v \in \mathcal{V}.
$$
\end{definition}

\begin{center}
\textbf{Curvature}
\end{center}

\noindent
There is a canonical extension of such an operator to a collection of $\mathbb{C}$-linear maps,
\begin{equation} \label{eq:ConnectionExtensionToForms}
\nabla^{(k)} : \mathcal{V} \otimes_\mathcal{O} \textstyle{\bigwedge^k_\mathcal{O}} \mathcal{K} \to \mathcal{V} \otimes_\mathcal{O} \textstyle{\bigwedge^{k+1}_\mathcal{O}} \mathcal{K}, \qquad k\geq 0,
\end{equation}
defined through the $\lambda$-twisted Leibniz rule
$$
\nabla^{(k)}(v \otimes \omega) = \lambda v \otimes D\omega + \nabla(v) \wedge \omega, \qquad\text{for } \omega \in \textstyle{\bigwedge^k_\mathcal{O}} \mathcal{K}, v \in \mathcal{V}.
$$
It is an easy exercise to check that the compositions $\nabla^{(k+1)} \circ \nabla^{(k)}$ are linear over $\bigwedge^k_\mathcal{O} \mathcal{K}$, in the sense that
$$
\left( \nabla^{(k+1)} \circ \nabla^{(k)} \right) (v \otimes \omega) = \left( \nabla^{(1)} \circ \nabla^{(0)} \right) (v) \wedge \omega, \quad\text{for } v \in V, \omega \in \textstyle{\bigwedge^k_\mathcal{O}}.
$$
In particular, $\nabla^{(1)} \circ \nabla^{(0)}$ is $\mathcal{O}$-linear and can be considered as a global section of the sheaf of endomorphisms of $\mathcal{V}$ with values in $\bigwedge^2_\mathcal{O} \mathcal{K}$ through the standard duality isomorphism
\begin{align*}
\operatorname{Hom}_\mathcal{O} \!\left( \mathcal{V}, \mathcal{V} \otimes_\mathcal{O} \textstyle{\bigwedge^2_\mathcal{O}} \mathcal{K} \right) &\cong \operatorname{Hom}_\mathcal{O} \!\left( \mathcal{O}, \sEnd_\mathcal{O}(\mathcal{V}) \otimes_\mathcal{O} \textstyle{\bigwedge^2_\mathcal{O}} \mathcal{K} \right) \\
&\cong \Gamma \left( \mathfrak{X}, \sEnd_\mathcal{O}(\mathcal{V}) \otimes_\mathcal{O} \textstyle{\bigwedge^2_\mathcal{O}} \mathcal{K} \right).
\end{align*}

\begin{definition} \label{def:Flatness}
The \emph{curvature} of a $\lambda$-$D$-connection, $\nabla$, on $\mathcal{V}$ is
$$
\nabla^{(1)} \circ \nabla^{(0)} \in \Gamma\left( \mathfrak{X}, \sEnd_\mathcal{O}(\mathcal{V}) \otimes_\mathcal{O} \textstyle{\bigwedge^2_\mathcal{O}} \mathcal{K} \right).
$$
We say that $\nabla$ is \emph{flat} (or \emph{integrable}) if its curvature vanishes.
\end{definition}

Once again, we drop the superscripts and simply denote by $\nabla$ all of the operators \eqref{eq:ConnectionExtensionToForms}, so that flatness is the requirement $\nabla^2 = 0$.

\begin{center}
\textbf{Cohomology}
\end{center}

\noindent
A pair $(\mathcal{V}, \nabla)$ as above where $\nabla$ is flat carries with it an intrinsic notion of cohomology: namely, the hypercohomology of the complex of sheaves of $\mathbb{C}$-vector spaces
\begin{equation} \label{eq:ComplexComputingIntrinsicCohomology}
0 \longrightarrow \mathcal{V} \xrightarrow{\;\nabla\;} \mathcal{V} \otimes_\mathcal{O} \mathcal{K} \xrightarrow{\;\nabla\;} \mathcal{V} \otimes_\mathcal{O} \textstyle{\bigwedge^2_\mathcal{O}} \mathcal{K} \xrightarrow{\;\nabla\;} \cdots.
\end{equation}
The latter is a dg-module over the cdga $\left( \bigwedge^\bullet_\mathcal{O} \mathcal{K}, D\right)$, and hence its hypercohomology, $\mathbb{H}^\bullet \!\left( X, (\mathcal{V}, \nabla) \right)$, comes equipped with a $\mathbb{H}^\bullet \!\left( X, \left( \bigwedge^\bullet_\mathcal{O} \mathcal{K}, D\right) \right)$-module structure.

In the remainder, we will need not only these hypercohomology groups, but also the complex of $\mathbb{C}$-vector spaces that computes them: the complex of \emph{derived global sections} of \eqref{eq:ComplexComputingIntrinsicCohomology}, denoted by $\mathbb{R}\Gamma( X, (\mathcal{V}, \nabla) )$. If the sheaves $\mathcal{V} \otimes_\mathcal{O} \bigwedge^\bullet_\mathcal{O} \mathcal{K}$ are acyclic, this is simply the complex of global sections of \eqref{eq:ComplexComputingIntrinsicCohomology},
$$
\mathbb{R}\Gamma(\mathcal{V}, \nabla) = \big( \Gamma(X, \mathcal{V} \otimes_\mathcal{O} \textstyle{\bigwedge^\bullet_\mathcal{O} \mathcal{K}}), \nabla \big).
$$
Otherwise\footnote{In the smooth setting, all of our sheaves are fine, hence acyclic, and this extra complication does not show up; some care must be exercised in the holomorphic category though.}, it is the complex
$$
\mathbb{R}\Gamma(\mathcal{V}, \nabla) = \big( \Gamma(X, \mathcal{I}^\bullet), \widetilde\nabla \big)
$$
of global sections of any acyclic resolution, $(\mathcal{V} \otimes_\mathcal{O} \textstyle{\bigwedge^\bullet_\mathcal{O}} \mathcal{K}, \nabla) \to (\mathcal{I}^\bullet, \widetilde\nabla)$, of \eqref{eq:ComplexComputingIntrinsicCohomology}.

\begin{center}
\textbf{Tensor products and duals}
\end{center}

\begin{definition}
Given two locally-free sheaves of coherent $\mathcal{O}$-modules equipped with $\lambda$-$D$-connections, $(\mathcal{V}_1, \nabla_1)$ and $(\mathcal{V}_2, \nabla_2)$, their \emph{tensor product} is given by
$$
(\mathcal{V}_1, \nabla_1) \otimes (\mathcal{V}_2, \nabla_2) = (\mathcal{V}_1 \otimes_\mathcal{O} \mathcal{V}_2, \nabla_1 \otimes 1 + 1 \otimes \nabla_2).
$$
\end{definition}

Curvature is additive under this operation, and so $\nabla_1 \otimes 1 + 1 \otimes \nabla_2$ is a flat $\lambda$-$D$-connection if so are $\nabla_1$ and $\nabla_2$.

\begin{definition} \label{definition:DualLambdaConnection}
Let $(\mathcal{V}, \nabla)$ be a locally-free sheaf of coherent $\mathcal{O}$-modules equipped with a $\lambda$-$D$-connection. We define its \emph{dual}, $(\mathcal{V}, \nabla)^\ast$, as the pair $(\mathcal{V}^\ast, \nabla^\ast)$, where the $\lambda$-$D$-connection $\nabla^\ast$ is defined by the equation
$$
\nabla^\ast (\eta) (v) + \eta (\nabla v) = \lambda D (\eta v), \qquad\text{for } \eta \in \mathcal{V}^\ast, v \in \mathcal{V}.
$$
\end{definition}

It is a good exercise to check that $\nabla^\ast$ satisfies the appropriate Leibniz rule, and that it is flat if $\nabla$ is.

\begin{center}
\textbf{Categories of $\lambda$-$D$-connections}
\end{center}

\noindent
For any fixed value of $\lambda$, there is an obvious notion of morphism between locally-free sheaves of coherent $\mathcal{O}$-modules equipped with a $\lambda$-$D$-connection; namely, a morphism from $(\mathcal{V}_1, \nabla_1)$ to $(\mathcal{V}_2, \nabla_2)$ is a morphism of sheaves of $\mathcal{O}$-modules, $f: \mathcal{V}_1 \to \mathcal{V}_2$, that commutes with the connections:
$$
\begin{tikzpicture}[x=90pt, y=40pt]
  \node (11) at (0,1) {$\mathcal{V}_1$};
  \node (12) at (1,1) {$\mathcal{V}_1 \otimes_\mathcal{O} \mathcal{K}$};
  \node (21) at (0,0) {$\mathcal{V}_2$};
  \node (22) at (1,0) {$\mathcal{V}_2 \otimes_\mathcal{O} \mathcal{K}$};
  \path[semithick,->]
    (11) edge node[scale=0.8,yshift=8pt] {$\nabla_1$} (12)
    (11) edge node[scale=0.8,xshift=8pt,pos=0.45] {$f$} (21)
    (12) edge node[scale=0.8,xshift=22pt,pos=0.45] {$f \otimes \operatorname{id}_\mathcal{K}$} (22)
    (21) edge node[scale=0.8,yshift=8pt] {$\nabla_2$} (22);
\end{tikzpicture}
$$
We will refer to the category determined by these arrows as the \emph{category of $\lambda$-$D$-connections}. In terms of the duals and tensor products, the above definition of morphisms is equivalent to setting
\begin{multline*}
\operatorname{Hom} \big( (\mathcal{V}_1, \nabla_1) , (\mathcal{V}_2, \nabla_2) \big) \\ = \ker \left\{ \Gamma(X, \mathcal{V}_1^\ast \otimes_\mathcal{O} \mathcal{V}_2) \xrightarrow{\;\;\nabla_1^\ast \otimes 1 + 1 \otimes \nabla_2\;\;} \Gamma(\mathcal{V}_1^\ast \otimes_\mathcal{O} \mathcal{V}_2 \otimes_\mathcal{O} \mathcal{K}) \right\}.
\end{multline*}

If we restrict our attention to the full \emph{subcategory of flat $\lambda$-$D$-connections}, these morphism sets have the following nice description in terms of the intrinsic cohomologies defined above:
\begin{equation} \label{eq:MorphismsFlatLambdaConnections}
\operatorname{Hom} \big( (\mathcal{V}_1, \nabla_1) , (\mathcal{V}_2, \nabla_2) \big) = \mathbb{H}^0 \big( X, (\mathcal{V}_1, \nabla_1)^\ast \otimes (\mathcal{V}_2, \nabla_2) \big).
\end{equation}
In fact, we can turn this subcategory into a \emph{$\mathbb{C}$-linear dg(=differential graded)-category}\footnote{A $\mathbb{C}$-linear dg-category is a $\mathbb{C}$-linear additive category in which morphisms assemble into complexes of $\mathbb{C}$-vector spaces.} by using, instead of just the zeroth hypercohomology group, the whole complex of derived global sections,
$$
\operatorname{Hom}^\bullet \!\big( (\mathcal{V}_1, \nabla_1) , (\mathcal{V}_2, \nabla_2) \big) = \mathbb{R}\Gamma \big( X, (\mathcal{V}_1, \nabla_1)^\ast \otimes (\mathcal{V}_2, \nabla_2) \big).
$$

\begin{center}
\textbf{$\mathbb{C}^\times$-actions}
\end{center}

\begin{lemma} \label{lemma:MultiplyingAConnection}
Let $\nabla$ be a $\lambda$-$D$-connection on $\mathcal{V}$, and $\mu \in \mathbb{C}$. Then $\mu\nabla$ is a $\lambda\mu$-$D$-connection on $\mathcal{V}$, which is flat if $\nabla$ is.
\end{lemma}

For $\lambda = 0$, this lemma defines an action of the multiplicative group $\mathbb{C}^\times$ on the category of $0$-$D$-connections that takes the full subcategory of flat $0$-$D$-connections into itself.

On the other hand, for $\lambda \neq 0$, it sets up an equivalence between the category of $1$-$D$-connections and that of $\lambda$-$D$-connections ---which is furthermore an equivalence of (dg-)categories between the corresponding full subcategories of flat objects.

\begin{center}
\textbf{Linear combinations}
\end{center}

\noindent
The following observation, whose proof is left as a trivial exercise for the reader, will be crucial in \S\ref{section:ASketchOfProof} below.

\begin{lemma} \label{lemma:LinearCombinationOfConnections}
For $i = 1, 2$, let $\nabla_i$ be a $\lambda_i$-$D$-connection on $\mathcal{V}$, and $\mu_i \in \mathbb{C}$. Then, $\mu_1\nabla_1 + \mu_2\nabla_2$ is a $(\mu_1\lambda_1 + \mu_2\lambda_2)$-$D$-connection on $\mathcal{V}$.
\end{lemma}

\begin{center}
\textbf{Examples}
\end{center}

\begin{example} \label{example:FlatConnections}
(cf. Example \ref{example:deRhamDifferential})
Let $E$ be a $C^\infty$ complex vector bundle\footnote{When we say vector bundle, we will always mean \emph{of finite rank}.} on $X$. A flat $1$-$d$-connection, $\nabla$, in the sense of Definition \ref{def:lambdaDConnection} is what is classically known as a \emph{flat connection} on $E$. The pair $(E, \nabla)$ then goes by the name of \emph{flat bundle} on $X$.

Its flat sections (i.e., those annihilated by the connection) form a locally constant sheaf of $\mathbb{C}$-vector spaces ---a \emph{local system}---, and the flat bundle can be completely recovered from the latter. Indeed, if $L$ is a local system, the vector bundle $L \otimes_\mathbb{C} \mathcal{A}^0_X$ has a natural flat connection defined by
$$
\nabla(l \otimes f) = l \otimes df, \qquad\text{for } l \in L, f \in \mathcal{A}^0_X.
$$

On the other hand, local systems encode representations of the fundamental group of the underlying topological space through the Riemann-Hilbert correspondence. This shows that, although their definition makes explicit use of the smooth structure of $X$, the category of flat bundles is independent of it: they are topological in nature.

Notice that the trivial line bundle, $\mathcal{A}^0_X$, admits a canonical flat connection, given by the de Rham differential \eqref{eq:deRhamDifferential} itself. The complex \eqref{eq:ComplexComputingIntrinsicCohomology} computing its intrinsic cohomology coincides with the de Rham complex \eqref{eq:deRhamComplex}, and so $\mathbb{H}^\bullet \!\left( X, (\mathcal{A}^0_X, d) \right)$ is just $H^\bullet_\mathrm{dR}(X)$. More generally, we call the intrinsic cohomology of a flat bundle, $(E, \nabla)$, its \emph{de Rham cohomology}, and we denote it by $H^\bullet_\mathrm{dR}(E, \nabla)$.
\end{example}

\begin{example} \label{example:HolomorphicStructures}
(cf. Example \ref{example:DolbeaultDifferentials})
Let $E$ be a $C^\infty$ complex vector bundle on $X$. A flat $1$-$\overline\partial$-connection, $\overline\partial_E$, on $E$ is called a \emph{holomorphic structure} on $E$. The name is justified by the Koszul-Malgrange integrability theorem \cite{MR0131882} ---a linear version of the celebrated Newlander-Nirenberg theorem \cite{MR0088770}---, which asserts that a $C^\infty$ complex vector bundle equipped with such an operator is the same data as that of a vector bundle in the holomorphic category: the holomorphic sections of $(E, \overline\partial_E)$ are those (smooth) sections of $E$ annihilated by $\overline\partial_E$.
\end{example}

\begin{example}
(cf. Example \ref{example:DolbeaultDifferentials})
The Dolbeault operator $\overline\partial$ is a holomorphic structure on the sheaf $\mathcal{A}^{p,0}_X$ of smooth complex-valued $p$-forms of type $(p,0)$ ---the holomorphic sections are the holomorphic $p$-forms, i.e., the sections of the sheaf $\Omega^p_X$ of Example \ref{example:HolomorphicdeRhamDifferential} above. The intrinsic cohomology of this pair displays the \emph{Dolbeault cohomology} of $X$:
$$
H_{\overline\partial}^{p,q}(X) = H^q(X, \Omega^p_X) \cong H^q \big( \Gamma(X, \mathcal{A}^{p,\bullet}_X), \overline\partial \big) \cong \mathbb{H}^q \big( X, (\mathcal{A}^{p,\bullet}_X, \overline\partial) \big).
$$
\end{example}

\begin{example} \label{example:HolomorphicFlatConnections}
(cf. Example \ref{example:HolomorphicdeRhamDifferential})
Let $\mathcal{E}$ be a holomorphic vector bundle on $X$. A flat $1$-$d_h$-connection, $\nabla$, on $\mathcal{E}$ is called a \emph{holomorphic flat connection} on $\mathcal{E}$. We refer to such a pair, $(\mathcal{E}, \nabla)$, as a \emph{flat holomorphic bundle}.
\end{example}

\begin{example} \label{example:SmoothAndHolomorphicFlatConnections}
We can tie together Examples \ref{example:HolomorphicStructures}
and \ref{example:HolomorphicFlatConnections} to obtain a holomorphic description of the flat bundles of Example \ref{example:FlatConnections}. Indeed, let $(E, \nabla)$ be a flat bundle on $X$. After separating the connection into its $(0,1)$ and $(1,0)$ parts,
$$
d'' : E \to E \otimes_{\mathcal{A}^0_X} \mathcal{A}^{0,1}_X, \qquad d' : E \to E \otimes_{\mathcal{A}^0_X} \mathcal{A}^{1,0}_X,
$$
the flatness condition, $\nabla^2 = 0$, splits up into the following three equations:
$$
\left( d'' \right)^2 = 0, \quad d''d' + d'd'' = 0, \quad \left( d' \right)^2 = 0.
$$
The first one makes $d''$ into a holomorphic structure on $E$; the second one ensures that $d'$ induces a $1$-$d_h$-connection, $\widetilde{d'}$, on the holomorphic vector bundle $\mathcal{E} = (E, d'')$; the third equation expresses the flatness of the latter.

Moreover, the complexes of derived global sections,
$\mathbb{R}\Gamma(X, (E, \nabla))$ and $\mathbb{R}\Gamma(X, (\mathcal{E}, \widetilde{d'}))$, are quasi-isomorphic ---that is, there is a morphism of complexes between them that induces an isomorphism of their cohomologies. Hence the de Rham cohomology of a flat bundle can also be calculated holomorphically.
\end{example}

\begin{example} \label{example:HiggsFields}
Let $\mathcal{F}$ be a holomorphic vector bundle on $X$. The term \emph{Higgs field} was coined by Hitchin \cite{MR887284} to describe a flat $0$-$d_h$-connection, $\phi$, on $\mathcal{F}$ ---the pair $(\mathcal{F}, \phi)$ being referred to as a \emph{Higgs bundle}. It is worth noting that, since $\lambda=0$, the operator $\phi$ is $\mathcal{O}_X$-linear, and its flatness is usually written as $\phi \wedge \phi = 0$, where $\phi \wedge \phi$ is the composition
$$
\mathcal{F} \xrightarrow{\;\;\phi\;\;} \mathcal{F} \otimes_{\mathcal{O}_X} \Omega^1_X \xrightarrow{\;\phi \otimes \mathrm{id}\;} \mathcal{F} \otimes_{\mathcal{O}_X} \Omega^1_X \otimes_{\mathcal{O}_X} \Omega^1_X \xrightarrow{\;\mathrm{id}\otimes(- \wedge -)\;} \mathcal{F} \otimes_{\mathcal{O}_X} \Omega^2_X.
$$

The intrinsic cohomology of a Higgs bundle, $(\mathcal{F}, \phi)$, is known as its \emph{Dolbeault cohomology}, $H^\bullet_\mathrm{Dol}(\mathcal{F}, \phi)$.
\end{example}

\begin{example} \label{example:HiggsFieldsSmoothDescription}
We can also give a description of Higgs bundles in the smooth category. Consider triples, $(F, \overline\partial_F, \theta)$, of a $C^\infty$ complex vector bundle equipped with a flat $1$-$\overline\partial$-connection and a flat $0$-$\partial$-connection subject to the condition
$$
\overline\partial_F \theta + \theta \overline\partial_F = 0;
$$
that is, such that $D'' = \overline\partial_F + \theta$ squares to zero. As in Example \ref{example:HolomorphicStructures}, the pair $(F, \overline\partial_F)$ encodes a holomorphic vector bundle, $\mathcal{F}$. Because of the commutativity condition above, the operator $\theta$ descends to a flat $0$-$d_h$-connection, $\phi$, on $\mathcal{F}$. The Dolbeault cohomology, $H^\bullet_\mathrm{Dol}(\mathcal{F}, \phi)$, of this Higgs bundle also appears as the hypercohomology of the complex
$$
0 \longrightarrow F \xrightarrow{\;D''\;} F \otimes_{\mathcal{A}^0_X} \mathcal{A}^1_X \xrightarrow{\;D''\;} F \otimes_{\mathcal{A}^0_X} \mathcal{A}^2_X \xrightarrow{\;D''\;} \cdots.
$$

The reader should convince himself or herself that, even though $D''$ is not a connection but a sum of two connections, the constructions of tensor products, duals and (dg-)morphisms between Higgs bundles in this description can be given the exact same definition as above.
\end{example}

\begin{example} \label{example:ChernCharactersAndClasses}
All of our previous examples have concentrated on flat connections; however, those with nonvanishing curvature have their importance too. If a $C^\infty$ complex vector bundle, $E$, on $X$ admits a $1$-$d$-connection, $\nabla$, its \emph{Chern characters} can be computed in terms of the curvature of $\nabla$ as follows:
$$
\operatorname{ch}_k(E) = \left[ \frac{1}{k!} \left(\frac{i}{2\pi}\right)^k \operatorname{Tr} \Big( \underbrace{\nabla^2 \wedge \ldots \wedge \nabla^2}_\text{$k$ copies} \Big) \right] \in H_\mathrm{dR}^{2k}(X), \qquad k \geq 0.
$$
Among these topological invariants of $E$, we find its \emph{rank} and its \emph{degree}:
$$
\operatorname{rk}(E) = \operatorname{ch}_0(E), \qquad \operatorname{deg}(E) = \operatorname{ch}_1(E) . [\omega]^{\operatorname{dim}X-1}.
$$
Here the \emph{intersection product} of cohomology classes, $[\alpha]\in H^{k}(X, \mathbb{C})$ and $[\beta] \in H^{\operatorname{dim}_\mathbb{R} X-k}(X, \mathbb{C})$, of complementary degree is defined as the integral of the wedge product of the differential forms representing them:
$$
[\alpha].[\beta] = \int_X \alpha \wedge \beta.
$$

Another, perhaps more well-known, set of invariants of $E$ are its \emph{Chern classes}. They can be written as linear combinations of powers of the Chern characters of $E$, and hence encode essentially the same information (the rank being excepted). For example,
$$\operatorname{c}_0(E) = 1, \quad \operatorname{c}_1(E) = \operatorname{ch}_1(E), \quad \operatorname{c}_2(E) =  \frac{1}{2} \operatorname{ch}_1(E)^2 - \operatorname{ch}_2(E).$$

It is important to keep in mind that these invariants are the same independently of which $1$-$d$-connection on $E$ is used to compute them. In particular, if $E$ admits a flat $1$-$d$-connection, all of its Chern characters and classes vanish, even if calculated using a nonflat $1$-$d$-connection.
\end{example}

\section{The correspondence between Higgs bundles and local systems}
\label{section:CorrespondenceHiggsBundlesAndLocalSystems}

Let $X$ be a compact K{\"a}hler manifold of complex dimension $n$. The \emph{nonabelian Hodge theorem} of Simpson's establishes a (dg-)equivalence between:
\begin{itemize}
\item the (dg-)category of flat holomorphic bundles, $(\mathcal{E}, \nabla)$, on $X$ ---which, recall, can be reconstructed from their local systems of flat sections (cf. Examples \ref{example:FlatConnections} and \ref{example:SmoothAndHolomorphicFlatConnections})---, and
\item a certain full subcategory of the (dg-)category of Higgs bundles, $(\mathcal{F}, \phi)$, on the same manifold.
\end{itemize}
This subcategory is specified by two conditions on objects. The first one is purely topological.

\begin{description}
\item[\textbf{Condition 1:}] the components of the first and second Chern characters of $\mathcal{F}$ along the K{\"a}hler class, $[\omega] \in H^2(X, \mathbb{C})$, of $X$ vanish:
$$
\operatorname{ch}_1(\mathcal{F}) . [\omega]^{n-1} = 0 = \operatorname{ch}_2(\mathcal{F}) . [\omega]^{n-2}.
$$
Equivalently, the first and second Chern classes of $\mathcal{F}$ vanish along $[\omega]$.
\end{description}
The second condition involves both the holomorphic structure of the bundle, and the Higgs field. In order to state it, we need the next definition. In it, $\mu$ is a numerical invariant of vector bundles known as \emph{slope} ---it is simply the quotient of their degree by their rank---, and a holomorphic subbundle $\mathcal{F}' \subset \mathcal{F}$ is \emph{$\phi$-invariant} (one also says that it is \emph{preserved by the Higgs field}) if $\phi(\mathcal{F}') \subset \mathcal{F}' \otimes \Omega^1_X$.
\begin{definition} \label{definition:Stability}
A Higgs bundle $(\mathcal{F}, \phi)$ is said to be \emph{stable}\footnote{We remind the reader that this is the notion of \emph{slope stability}.} (resp., \emph{semistable}) if for every $\phi$-invariant, holomorphic subbundle $\mathcal{F}' \subset \mathcal{F}$ we have $\mu(\mathcal{F}') < \mu(\mathcal{F})$ (resp., $\mu(\mathcal{F}') \leq \mu(\mathcal{F})$).
\end{definition}
\begin{description}
\item[\textbf{Condition 2:}] $(\mathcal{F}, \phi)$ is an iterated extension\footnote{It is the fact that the category of Higgs bundles is abelian that allows us to talk about extensions.} of \emph{polystable} ---that is, direct sums of stable--- Higgs bundles, each of which satisfies Condition 1.
\end{description}
In case $X$ is a smooth complex projective variety, we can give a cleaner characterization.
\begin{description}
\item[\textbf{Condition 2$^\prime$:}] $(\mathcal{F}, \phi)$ is semistable.
\end{description}

\begin{proposition}
Let $X$ be a smooth complex projective variety, and $(\mathcal{F}, \phi)$ a Higgs bundle on it satisfying Condition 1 above. Then, $(\mathcal{F}, \phi)$ satisfies Condition 2 if and only if it satisfies Condition 2$\,^\prime$\footnote{The reader might have noticed that Condition 2 implies Condition 1. Condition 2$^\prime$ is, however, independent of Condition 1, and the equivalence between Conditions 2 and 2$^\prime$ does not hold in the absence of Condition 1.}.
\end{proposition}

But why should we expect such a correspondence between Higgs bundles and local systems? Our discussion in \S\ref{section:ZoologyOfConnections} provides an answer: holomorphic flat connections are flat $1$-$d_h$-connections, and Higgs fields are flat $0$-$d_h$-connections; flat $\lambda$-$d_h$-connections interpolate between them.

Lemma \ref{lemma:MultiplyingAConnection} offers us a way of constructing a family of flat $\lambda$-$d_h$-connections, one for each $\lambda \in \mathbb{C}^\times$, starting from a holomorphic flat connection. Unfortunately, as $\lambda \to 0$ this leaves us with the zero Higgs field, no matter what the original holomorphic flat connection was.

Although this na{\"\i}ve approach fails to realize Simpson's correspondence, the idea of deforming a flat holomorphic bundle to a Higgs bundle through a family of holomorphic vector bundles equipped with a flat $\lambda$-$d_h$-connection can still be made to work (up to a point: see \S\ref{section:Extensions}). The key idea is to allow the holomorphic structure on the underlying $C^\infty$ bundle to change with the parameter $\lambda$ in a controlled way.

\section{A sketch of proof}
\label{section:ASketchOfProof}

Let $E$ be a $C^\infty$ complex vector bundle on $X$. Recall (cf. Example \ref{example:FlatConnections}) that we can endow $E$ with a flat bundle structure by specifying
\begin{description}
\item[\hspace{10pt}$-$] a flat $1$-$\overline\partial$-connection, $d''$, and
\item[\hspace{10pt}$-$] a flat $1$-$\partial$-connection, $d'$,
\item[\hspace{10pt}$-$] satisfying the commutation relation $d'' d' + d' d'' = 0.$
\end{description}
The sum, $D = d' + d''$, is then a flat $1$-$d$-connection ---we have $D^2 = 0$. Similarly, to give a Higgs bundle structure on $E$ we need (cf. Example \ref{example:HiggsFieldsSmoothDescription})
\begin{description}
\item[\hspace{10pt}$-$] a flat $1$-$\overline\partial$-connection\footnote{A certain overloading of the notation is difficult to avoid at this point without making it overly cumbersome. Hopefully the reader will be able to tell the differential $\overline\partial$ from the connection $\overline\partial$ from the context. The same will unfortunately happen with other symbols.}, $\overline\partial$, and
\item[\hspace{10pt}$-$] a flat $0$-$\partial$-connection, $\theta$,
\item[\hspace{10pt}$-$] subject to the condition $\overline\partial \theta + \theta \overline\partial = 0$.
\end{description}
The sum, $D'' = \overline\partial + \theta$, no longer has a neat description as some kind of connection, but still satisfies the integrability condition $(D'')^2 = 0$.

Our aim below will be to pass between these two kinds of structure by the clever use of metrics on $E$. The natural notion of metric on a $C^\infty$ complex vector bundle is the following.

\subsection{Hermitian metrics}
\label{section:HermitianMetrics}

\begin{definition}
A \emph{Hermitian metric} on a $C^\infty$ complex vector bundle, $E$, is a smooth, $\mathbb{C}$-linear map,
$K: E \otimes_{\mathcal{A}^0_X} \overline{E} \to \mathcal{A}^0_X$, such that, at every point $p \in X$, the following two conditions hold:
\begin{enumerate}
\item $K_p (e, \overline{e'}) = \overline{K_p (e', \overline{e})}$ for all $e, e' \in E_p$, and
\item $K_p(e, \overline{e}) > 0$ for all $e \in E_p$.
\end{enumerate}
\end{definition}

An argument using partitions of unity proves that Hermitian metrics exist on any $C^\infty$ complex vector bundle.

\begin{center}
\textbf{Chern connections}
\end{center}

\begin{proposition} \label{proposition:ExistenceOfChernConnection}
Let $K$ be a Hermitian metric on a $C^\infty$ complex vector bundle, $E$. Given a holomorphic structure, $\overline\partial_E$, on $E$, there exists a unique $1$-$d$-connection on $E$, $\nabla$ ---called the \emph{Chern connection} of $(E, \overline\partial_E)$ with respect to $K$---, such that
\begin{enumerate}
\item the $(0,1)$ part of $\nabla$ is $\overline\partial_E$, and
\item $\nabla$ is a metric connection, i.e., $K(\nabla e, \overline{e'}) + K(e, \overline{\nabla e'}) = d\left( K(e, \overline{e'}) \right)$.
\end{enumerate}
Moreover, the $(1,0)$ part of this Chern connection is flat.
\end{proposition}

\begin{proof}
For $\{ e_\alpha \}_{\alpha=1}^r$ a local holomorphic frame for $E$, denote $K_{\alpha\beta} = K(e_\alpha, \overline{e_\beta})$. Since $\overline\partial_E (e_\alpha) = 0$, we have $\nabla(e_\alpha) = \sum_{\beta=1}^r \omega_{\alpha\beta} e_\beta$ with $\omega_{\alpha\beta} \in \mathcal{A}^{1,0}_X$. The metric condition then yields
$$
dK_{\alpha\beta} = \sum_{\gamma=1}^r \left( \omega_{\alpha\gamma} K_{\gamma\beta} + \overline{\omega_{\beta\gamma}} K_{\alpha\gamma} \right).
$$
The $(1,0)$ part of this equation ---which is the complex-conjugate of its $(0,1)$ part---, reads
$$
\partial K_{\alpha\beta} = \sum_{\gamma=1}^r \omega_{\alpha\gamma} K_{\gamma\beta};
$$
that is, the connection coefficients, $\omega_{\alpha\beta}$, are completely determined by the metric through the formula
$$
\omega_{\alpha\beta} = \sum_{\gamma=1}^r (\partial K_{\alpha\gamma}) (K^{-1})_{\gamma\beta}.
$$
We leave for the reader to show that the $(2,0)$ part of the curvature of $\nabla$ vanishes ---a local computation involving these connection coefficients. \qed
\end{proof}

\begin{corollary} \label{corollary:ComplexConjugateOfChernConnection}
Let $K$ be a Hermitian metric on a $C^\infty$ complex vector bundle, $E$. Given a flat $1$-$\partial$-connection, $\partial_E$, on $E$, there exists a unique $1$-$d$-connection on $E$, $\nabla$, such that
\begin{enumerate}
\item the $(1,0)$ part of $\nabla$ is $\partial_E$, and
\item $\nabla$ is a metric connection.
\end{enumerate}
Moreover, the $(0,1)$ part of $\nabla$ is a holomorphic structure on $E$.
\end{corollary}

\begin{proof}
The operator $\partial_E$ induces a holomorphic structure on the complex-conjugate bundle, $\overline{E}$, of $E$. Apply the previous proposition to the latter. \qed
\end{proof}

\begin{center}
\textbf{The $L^2$-product}
\end{center}

\noindent
A Hermitian metric, $K$, on a $C^\infty$ complex vector bundle, $E$, naturally induces a Hermitian inner product on the vector space of its global sections by integration
$$
(e, e')_{L^2} = \int_X K(e, \overline{e'}) \, \frac{\omega^n}{n!}, \qquad\text{for } e, e' \in \Gamma(X, E),
$$
against the volume form, $\omega^n/n!$, determined by the K{\"a}hler form, $\omega$. This can be extended to the whole algebra of (global) differential forms on $X$ with values in $E$ as follows. First, regard the Hermitian metric, $K$, on $E$ as an isomorphism $K^\flat: \overline{E} \to E^\ast$, given by the formula $K^\flat(\overline{e}) = K(-, \overline{e})$. Then use this isomorphism to define the following extension of the usual Hodge star operator:
$$
\ast_E : \mathcal{A}^{q,p}_X \otimes_{\mathcal{A}^0_X} \overline{E} \to \mathcal{A}^{n-p,n-q}_X \otimes_{\mathcal{A}^0_X} E^\ast, \quad \ast_E (\eta \otimes \overline{e}) = \ast\eta \otimes K^\flat(\overline{e}).
$$
The \emph{$L^2$-product} on $\Gamma(X, \mathcal{A}^\bullet_X \otimes_{\mathcal{A}^0_X} E)$ is then given by
\begin{equation} \label{eq:L2Product}
(\alpha, \beta)_{L^2} = \int_X \alpha \wedge \ast_E \overline\beta, \qquad\text{for } \alpha, \beta \in \Gamma(X, \mathcal{A}^\bullet_X \otimes_{\mathcal{A}^0_X} E),
\end{equation}
where $\wedge$ is the exterior product on the form part and the evaluation map, $E \otimes_{\mathcal{A}^0_X} E^\ast \to \mathcal{A}^0_X$, in the bundle part.

For $E = \mathcal{A}^0_X$ the trivial line bundle, this restricts to the well-known \emph{Hodge-Riemann pairing} of differential forms on $X$,
\begin{equation} \label{eq:HodgeRiemannPairing}
(\alpha, \beta)_{L^2} = \int_X \alpha \wedge \ast\overline\beta, \qquad\text{for } \alpha, \beta \in \Gamma(X, \mathcal{A}^\bullet_X),
\end{equation}
and just as the latter satisfies the so-called \emph{Hodge-Riemann bilinear relations}, there is a set of conditions satisfied by the former. Among them, we will point out that there is an extension of the Lefschetz decomposition on $\Gamma(X, \mathcal{A}^\bullet_X)$ to forms with values in $E$, and that it is orthogonal for the $L^2$-product ---just as the original Lefschetz decomposition is for \eqref{eq:HodgeRiemannPairing}.

\subsection{Harmonic bundles}
\label{section:HarmonicBundles}

Let $D = d' + d''$ be a flat bundle structure on $E$, and choose a Hermitian metric, $K$, on $E$. Using Proposition \ref{proposition:ExistenceOfChernConnection}, we construct a flat $1$-$\partial$-connection, $\delta'_K$, with the property that $\delta'_K + d''$ is a metric connection. Similarly, Corollary \ref{corollary:ComplexConjugateOfChernConnection} produces a flat $1$-$\overline\partial$-connection, $\delta''_K$, such that $d' + \delta''_K$ preserves the metric. We now define (cf. Lemma \ref{lemma:LinearCombinationOfConnections})
\begin{align*}
\partial_K &= \frac{d' + \delta'_K}{2} &\text{(a $1$-$\partial$-connection),} \\
\overline\partial_K &= \frac{d'' + \delta''_K}{2} &\text{(a $1$-$\overline\partial$-connection),} \\
\theta_K &= \frac{d' - \delta'_K}{2} &\text{(a $0$-$\partial$-connection),} \\
\overline\theta_K &= \frac{d'' - \delta''_K}{2} &\text{(a $0$-$\overline\partial$-connection).}
\end{align*}
Notice that none of these are, in general, flat. We can nevertheless combine them into
$$
D'_K = \partial_K + \overline\theta_K, \qquad\text{and}\qquad D''_K = \overline\partial_K + \theta_K.
$$
We have that $D = D'_K + D''_K$. This last operator, $D''_K$, is precisely of the type that defines a Higgs bundle structure on $E$; however, we have no reason to expect its \emph{pseudocurvature}\footnote{The prefix pseudo- reflects the fact that $D''_K$ is not a connection, but rather a sum of two connections.}, $\mathbf{G}_K = (D''_K)^2$, to vanish.

Suppose now that, instead of a flat bundle structure, $D$, we start with a Higgs bundle structure, $D'' = \overline\partial + \theta$, on $E$. The choice of a Hermitian metric, $K$, determines a flat $1$-$\partial$-connection, $\partial_K$, through Proposition \ref{proposition:ExistenceOfChernConnection}. We can also define a flat $0$-$\overline\partial$-connection, $\overline\theta_K$, as the adjoint of $\theta$ with respect to the metric\footnote{This is the \emph{minus} the dual of $\theta$ in the sense of Definition \ref{definition:DualLambdaConnection}.}:
$$
K(\overline\theta_K e, \overline{e'}) = K(e, \overline{\theta e'}).
$$
Set now
$$
D'_K = \partial_K + \overline\theta_K, \qquad\text{and}\qquad D_K = D'_K + D''.
$$
Once again, the curvature, $\mathbf{F}_K = D_K^2$, of $D_K$ is, in general, nonzero.

Notice that, in general, we cannot compose these constructions ---flatness of the initial operator is needed. The crucial observation is that, when either ---and hence both--- $\mathbf{F}_K = 0$ or $\mathbf{G}_K = 0$, these two constructions are not only composable, but inverses of each other in the following sense: if the first construction (resp., the second) yields $\mathbf{G}_K = 0$ (resp., $\mathbf{F}_K = 0$), then the second construction (resp., the first) recovers the original flat bundle structure (resp., Higgs bundle structure). We will then say that $D$ and $D''$ are \emph{$K$-related}, or that \emph{$K$ relates $D$ with $D''$}.

\begin{definition}
A \emph{harmonic bundle} on $X$ is a triple $(E, D, D'')$, where $D$ is a flat bundle structure on $E$ and $D''$ is a Higgs bundle structure on $E$, such that there is a Hermitian metric, $K$, on $E$ relating $D$ with $D''$.\end{definition}

\begin{center}
\textbf{Families of $\lambda$-$d_h$-connections}
\end{center}

\begin{proposition} \label{proposition:LambdaConnectionsFromHarmonicBundles}
Given a harmonic bundle on $X$, $(E, D, D'')$, there is a family, $(\mathcal{E}_\lambda, \nabla_\lambda)$, of holomorphic vector bundles on $X$ equipped with a flat $\lambda$-$d_h$-connection such that $(\mathcal{E}_1, \nabla_1) = (E, D)$ and $(\mathcal{E}_0, \nabla_0) = (E, D'')$.
\end{proposition}

\begin{proof}
The family is given by the operator $D_\lambda = \lambda D' + D''$ on $E$. Indeed, for each $\lambda \in \mathbb{C}$,
\begin{description}
\item[\hspace{10pt}$-$] its $(0,1)$ part, $\overline\partial + \lambda \overline\theta$, is a flat $1$-$\overline\partial  $-connection, and
\item[\hspace{10pt}$-$] its $(1,0)$ part, $\lambda \partial + \theta$, is a flat $\lambda$-$\partial$-connection.
\end{description}
Since $D_\lambda^2 = 0$, these two connections commute, and so $\lambda \partial + \theta$ induces a flat $\lambda$-$d_h$-connection, $\nabla_\lambda$, on the holomorphic vector bundle $\mathcal{E}_\lambda = (E, \overline\partial + \lambda \overline\theta)$. \qed
\end{proof}

These are precisely the families we were looking for. The question now is which flat (resp., Higgs) bundles admit a Hermitian metric, $K$, such that $\mathbf{G}_K = 0$ (resp., $\mathbf{F}_K = 0$), and whether the choice of such a metric is unique. We will tackle these issues in \S\ref{section:FromFlatBundlesToHarmonicBundles} (resp., \S\ref{section:FromHiggsBundlesToHarmonicBundles}).

\begin{center}
\textbf{Tensor products and duals}
\end{center}

\noindent
Suppose $(E, D_E, D''_E)$ and $(F, D_F, D''_F)$ are harmonic bundles, and let $K_E$ and $K_F$ be Hermitian metrics relating $D_E$ with $D''_E$, and $D_F$ with $D''_F$, respectively. Then the product metric, $K_E \otimes K_F$, on $E \otimes F$ defined by
$$
(K_E \otimes K_F) (e \otimes f, \overline{e' \otimes f'}) = K_E(e, \overline{e'}) K_F(f, \overline{f'}), \quad\text{for } e, e' \in E, f, f' \in F
$$
relates $D_E \otimes 1 + 1 \otimes D_F$ with $D''_E \otimes 1 + 1 \otimes D''_F$. In other words, the triple
$$
(E \otimes F, D_E \otimes 1 + 1 \otimes D_F, D''_E \otimes 1 + 1 \otimes D''_F)$$
is a harmonic bundle.

Similarly, there is a metric, $K_E^\ast$, on $E^\ast$ ---called the metric dual to $K_E$--- relating $D_E^\ast$ with $(D''_E)^\ast$, and making $(E^\ast, D_E^\ast, (D''_E)^\ast)$ into a harmonic bundle.

\begin{center}
\textbf{The category of harmonic bundles}
\end{center}

\noindent
It is clear that one should want to define morphisms between harmonic bundles in a such a way that the forgetful maps $(E, D, D'') \mapsto (E, D)$ and $(E, D, D'') \mapsto (E, D'')$ become forgetful \emph{functors}. Thus, the obvious guesses would be to set
\begin{align*}
\operatorname{Hom} \big( (E, D_E, D''_E), (F, D_F, D''_F) \big) &= \operatorname{Hom} \big( (E, D_E), (F, D_F) \big),
\intertext{or}
\operatorname{Hom} \big( (E, D_E, D''_E), (F, D_F, D''_F) \big) &= \operatorname{Hom} \big( (E, D''_E), (F, D''_F) \big).
\end{align*}
That these two choices are one and the same follows from applying this next lemma to bundles of the form $E^\ast \otimes F$ (cf. \eqref{eq:MorphismsFlatLambdaConnections}).

\begin{lemma} \label{lemma:IsomorphismDeRhamDolbeaultDegreeZero}
Let $(E, D, D'')$ be a harmonic bundle on $X$, and $e \in E$. Then $De = 0$ if and only if $D''e = 0$ ---and hence $D'e = 0$. Thus there is a natural isomorphism $H^0_\mathrm{dR}(E, D) \cong H^0_\mathrm{Dol}(E, D'')$.
\end{lemma}

An added benefit of this definition of morphisms between harmonic bundles is that it makes the forgetful functors alluded to above fully-faithful, and hence equivalences onto their essential images.

\medskip
The dg-enhancement is only a tad more delicate to define: we just need to determine what the ``derived global sections'', $\mathbb{R}\Gamma(X, (E, D, D''))$, should be, and then set
$$
\operatorname{Hom} \big( (E, D_E, D''_E), (F, D_F, D''_F) \big) = \mathbb{R}\Gamma(X, (E, D_E, D''_E)^\ast \otimes (F, D_F, D''_F)).
$$
The cue for the correct definition comes from abelian Hodge theory: there, harmonic forms realize both de Rham and Dolbeault cohomology classes; here, we need to consider harmonic forms \emph{with values in $E$}. The latter appear as the cohomology classes of the complex of $\mathbb{C}$-vector spaces,
$$
\mathbb{R}\Gamma(X, (E, D, D'')) = \ker \big\{ D': \Gamma(E \otimes_{\mathcal{A}^0_X} \mathcal{A}^\bullet_X) \to \Gamma(E \otimes_{\mathcal{A}^0_X} \mathcal{A}^{\bullet+1}_X) \big\},
$$
with the restriction of $D''$ to it as differential. Then, the following result ensures that the forgetful functors are also fully-faithful in the dg-sense.

\begin{proposition}
Let $(E, D, D'')$ be a harmonic bundle on $X$. There are natural quasi-isomorphisms
\begin{align*}
\mathbb{R}\Gamma(X, (E, D, D'')) &\longrightarrow \mathbb{R}\Gamma(X, (E, D)), \\
\mathbb{R}\Gamma(X, (E, D, D'')) &\longrightarrow \mathbb{R}\Gamma(X, (E, D'')). \end{align*}
In particular, $H^\bullet_\mathrm{dR}(E, D) \cong H^\bullet_\mathrm{Dol}(E, D'')$.
\end{proposition}

The easiest case of this proposition should be quite familiar to the reader: applied to the harmonic bundle $(\mathcal{A}^0_X, d, \overline\partial)$, it reproduces the well-known isomorphisms
$$
H^\bullet_\mathrm{dR}(X) \cong \bigoplus_{p+q=\bullet} H^{p,q}_{\overline\partial}(X).
$$
Not only that, but the proof (which also contains within itself that of Lemma \ref{lemma:IsomorphismDeRhamDolbeaultDegreeZero}) also parallels the classical case: one needs to talk about K{\"a}hler identities, Laplacians and ---as we already announced--- harmonic forms.

\subsection{From flat bundles to harmonic bundles}
\label{section:FromFlatBundlesToHarmonicBundles}

The equation we are interested in solving, $\mathbf{G}_K = 0$, is overdetermined: the orthogonality of the Lefschetz decomposition on $\Gamma(X, \mathcal{A}^\bullet_X \otimes_{\mathcal{A}^0_X} E)$ with respect to the $L^2$-product \eqref{eq:L2Product}, implies the relation
\begin{equation} \label{eq:HodgeRiemannBilinearRelations1}
\operatorname{Tr}\left( \mathbf{G}_K \wedge \mathbf{G}_K \right) . [\omega]^{n-2} = C_1 \| \mathbf{G}_K \|_{L^2}^2 + C_2 \| \mathrm{\Lambda} \mathbf{G}_K \|_{L^2}^2,
\end{equation}
where $C_1$ and $C_2$ are nonzero constants, and $\mathrm{\Lambda}$ denotes the formal adjoint to the Lefschetz operator, $L = \omega \wedge -$. Furthermore, the \emph{pseudo-Chern character} on the left hand side of this equality vanishes.

\begin{lemma} \label{lemma:VanishingSecondPseudoChernClass}
$\operatorname{Tr}\left( \mathbf{G}_K \wedge \mathbf{G}_K \right) . [\omega]^{n-2} = 0$.
\end{lemma}

\begin{proof}
Consider the operator $D_\lambda = \lambda D'_K + D''_K$ (cf. proof of Proposition \ref{proposition:LambdaConnectionsFromHarmonicBundles}). For each $\lambda \in \mathbb{C}$, we can construct a $1$-$d$-connection, $B_\lambda$, out of it by rescaling its $(1,0)$ part:
$$
B_\lambda = \frac{1}{\lambda} D_\lambda^{1,0} + D_\lambda^{0,1} = \left( \partial_K + \frac{\theta_K}{\lambda} \right) + \left( \overline\partial _K+ \lambda \overline\theta_K \right).
$$
Its curvature computes the Chern classes of $E$. The latter are, however, zero since $E$ admits a flat $1$-$d$-connection ---$D$ (cf. Example \ref{example:ChernCharactersAndClasses}). Then, taking into account that $\omega$ is of type $(1,1)$, we have
$$
0 = \int_X \operatorname{Tr} \left( B_\lambda^2 \wedge B_\lambda^2 \right) \wedge \omega^{n-2} = \frac{1}{\lambda^2} \int_X \operatorname{Tr} \left( D_\lambda^2 \wedge D_\lambda^2 \right) \wedge \omega^{n-2}.
$$
The claim follows by letting $\lambda \to 0$, since, in that case, $D_\lambda \to D''_K$. \qed
\end{proof}

It is therefore enough to demand the component of the pseudocurvature along the K{\"a}hler class to vanish.

\begin{definition}
A Hermitian metric, $K$, on a flat bundle, $(E, D)$, is called \emph{harmonic} if $\mathrm{\Lambda}\mathbf{G}_K = 0$.
\end{definition}

The origin of the term `harmonic' here ---and ultimately in the name `harmonic bundle'--- lies in the fact that the equation $\mathrm{\Lambda}\mathbf{G}_K = 0$ is equivalent to the map
$$\Phi_K: \widetilde{X} \to {\fontfamily{cmss}\selectfont\textbf{GL}}_r(\mathbb{C}) / {\fontfamily{cmss}\selectfont\textbf{U}}(r)$$
classifying the metric $K$ being harmonic as a $\pi_1 X$-equivariant\footnote{The $\pi_1 X$-action on ${\fontfamily{cmss}\selectfont\textbf{GL}}_r(\mathbb{C}) / {\fontfamily{cmss}\selectfont\textbf{U}}(r)$ is that given by the monodromy representation of the flat bundle, $(E, D)$; that on the universal covering space, $\widetilde{X}$, of $X$ comes from deck transformations.} map between Riemannian manifolds \cite{MR965220} ---meaning that it minimizes a certain \emph{energy functional} on the space of all such maps.

\medskip
The following theorem is the first of two results (the other being Theorem \ref{theorem:PolystableHiggsBundlesAreHermitianYangMills}) that contain within themselves the bulk of the analysis needed to prove the nonabelian Hodge theorem. For a brief history of it, we refer the reader to the comments preceding \cite[Lemma 1.1]{MR1179076}.

\begin{theorem} \label{theorem:SemisimpleFlatBundlesAreHarmonic}
A flat bundle admits a harmonic metric if and only if it is semisimple. Moreover, such a harmonic metric is unique up to scalars.
\end{theorem}

Here semisimplicity has the usual meaning: a flat bundle is said to be \emph{semisimple} if every subbundle preserved by the connection (the natural notion of subobject in the category of flat bundles) is a direct summand. This condition is equivalent to semisimplicity of the representation of $\pi_1 X$ associated to it through the Riemann-Hilbert correspondence (cf. Example \ref{example:FlatConnections}).

The uniqueness up to scalars is the best we could have hoped for, since the constructions in \S\ref{section:HarmonicBundles} are invariant under multiplication of the Hermitian metric by a constant factor.

\begin{corollary} \label{corollary:EquivalenceFlatAndHarmonicBundles}
There is a (dg-)equivalence between the (dg-)category of harmonic bundles on $X$ and that of semisimple flat bundles on $X$.
\end{corollary}

\begin{proof}
We already established the fully-faithfulness of the forgetful functor in \S\ref{section:HarmonicBundles}. Essential surjectivity follows from \eqref{eq:HodgeRiemannBilinearRelations1}, Lemma \ref{lemma:VanishingSecondPseudoChernClass} and Theorem \ref{theorem:SemisimpleFlatBundlesAreHarmonic}.
\qed
\end{proof}

\subsection{From Higgs bundles to harmonic bundles}
\label{section:FromHiggsBundlesToHarmonicBundles}

Solving the equation $\mathbf{F}_K = 0$ goes more or less along the same lines as solving $\mathbf{G}_K = 0$. To start with, we also have the relation 
$$
\operatorname{Tr}\left( \mathbf{F}_K \wedge \mathbf{F}_K \right) . [\omega]^{n-2} = C_1 \| \mathbf{F}_K \|_{L^2}^2 + C_2 \| \mathrm{\Lambda} \mathbf{F}_K \|_{L^2}^2,
$$
deduced from the orthogonality of the Lefschetz decomposition. However, the second Chern character on the left hand side is no longer zero for an arbitrary Higgs bundle ---we will have to explicitly ask for it to vanish.

The kind of metrics on Higgs bundles that have been studied in the literature is the following.

\begin{definition}
A Hermitian metric, $K$, on a Higgs bundle, $(E, D'')$, is called \emph{Hermitian--Yang--Mills} if $\mathrm{\Lambda}\mathbf{F}_K = \gamma\operatorname{id}_E$ for some constant, $\gamma$.
\end{definition}

The constant $\gamma$ depends on background geometric data ---the manifold, $X$, the K{\"a}hler class, $[\omega]$, and the bundle, $E$---, but is independent of the metric.

\begin{lemma}
If $K$ is a Hermitian--Yang--Mills metric on $(E, D'')$, then
$$
\gamma = -\frac{2\pi i}{(n-1)!\operatorname{vol(X)}} \, \mu(E).
$$
\end{lemma}

\begin{proof}
Taking the trace of the Hermitian--Yang--Mills equation and integrating against the volume form, $\omega^n / n!$, we have
$$
\int_X \mathrm{\Lambda} \operatorname{Tr}\mathbf{F}_K \, \frac{\omega^n}{n!} = \gamma \operatorname{rk}(E)\operatorname{vol}(X).
$$
On the other hand,
$$
\int_X \mathrm{\Lambda} \operatorname{Tr}\mathbf{F}_K \, \frac{\omega^n}{n!} = \int_X  \operatorname{Tr}\mathbf{F}_K \wedge \frac{\omega^{n-1}}{(n-1)!} = - \frac{2\pi i}{(n-1)!} \operatorname{deg}(E),
$$
where the first equality results from the fact that
$$
\operatorname{Tr}\mathbf{F}_K - \frac{1}{n} \left( \mathrm{\Lambda} \operatorname{Tr}\mathbf{F}_K \right) \omega$$
is a primitive $(1,1)$ form. The definition of slope, $\mu(E) = \operatorname{deg}(E) / \operatorname{rk}(E)$, completes the proof of the claim. \qed
\end{proof}

The following is the second of our black boxes. Its history is discussed in the comments below \cite[Theorem 1]{MR1179076}.

\begin{theorem} \label{theorem:PolystableHiggsBundlesAreHermitianYangMills}
A Higgs bundle admits a Hermitian--Yang--Mills metric if and only if it is polystable. Such a metric is unique up to scalars.
\end{theorem}

\begin{corollary} \label{corollary:EquivalenceHiggsAndHarmonicBundles}
There is a (dg-)equivalence between the (dg-)category of harmonic bundles and that of polystable Higgs bundles with
$$
\operatorname{ch}_1(\mathcal{F}) . [\omega]^{n-1} = 0 = \operatorname{ch}_2(\mathcal{F}) . [\omega]^{n-2}.
$$
\end{corollary}

Here the vanishing of the first Chern character along the K{\"a}hler class ensures that the constant $\gamma$ in the Hermitian--Yang--Mills equation is zero.

\medskip
Putting together Corollaries \ref{corollary:EquivalenceFlatAndHarmonicBundles} and \ref{corollary:EquivalenceHiggsAndHarmonicBundles}, we arrive at a conclusion that is very close to the full statement of the nonabelian Hodge theorem of \S\ref{section:CorrespondenceHiggsBundlesAndLocalSystems}.

\begin{corollary} \label{corollary:EquivalenceFlatAndHiggsBundles}
There is a (dg-)equivalence between the (dg-)category of semi\-simple flat bundles and that of polystable Higgs bundles with
$$
\operatorname{ch}_1(\mathcal{F}) . [\omega]^{n-1} = 0 = \operatorname{ch}_2(\mathcal{F}) . [\omega]^{n-2}.
$$
\end{corollary}

\subsection{Extensions}
\label{section:Extensions}

A cursory look at the statements of the nonabelian Hodge theorem and of Corollary \ref{corollary:EquivalenceFlatAndHiggsBundles} makes it clear that the objects in the former are just (iterated) extensions of those in the latter.

However, Corollary \ref{corollary:EquivalenceFlatAndHiggsBundles} is the most general correspondence between local systems and Higgs bundles that can be realized through families of holomorphic vector bundles equipped with a flat $\lambda$-$d_h$-connection ---in particular, extensions of harmonic bundles are no longer harmonic.

The reason why we have fastidiously insisted that all of our categories be, in fact, differential graded categories is that this lifting of the equivalence in Corollary \ref{corollary:EquivalenceFlatAndHiggsBundles} to the corresponding categories of extensions follows from purely formal arguments. That is, one can prove that every time we have two equivalent dg-categories, their categories of extensions\footnote{We are purposefully avoiding here giving the definition of what extensions in a dg-category are. The interested reader can consult \S 3 of Simpson's \cite{MR1179076}.} are also equivalent in the dg-sense.

This finishes our sketch of the proof of Simpson's theorem. We encourage the reader wishing to fill in the details we left out to consult the original source \cite{MR1179076}.

\section{The hyperk\"ahler manifold}
\label{section:TheHyperkaehlerManifold}

The most important geometric consequence of the correspondence in \S\ref{section:CorrespondenceHiggsBundlesAndLocalSystems} is, arguably, the existence of a manifold with several modular intepretations at once.  This manifold can be equipped with non-isomorphic complex structures, each having an interpretation as a moduli space of geometric structures related to vector bundles.

\begin{theorem} \emph{(Hitchin, \cite{MR887284}; Simpson, \cite{MR944577})}
Let $X$ be a smooth, compact, connected complex curve of genus $g\geq2$.  For each positive integer $r$, there exists a noncompact singular variety $\mathcal{H}_X(r)$ which is a complex variety for three distinct complex structures $I$, $J$, and $K$ satisfying $I^2=J^2=K^2=IJK=-1$.  The respective complex varieties are denoted $\mathcal H^I_X(r)$, $\mathcal H^J_X(r)$, and $\mathcal H^K_X(r)$, and they have the following properties:
\begin{enumerate}
\item $\mathcal H^I_X(r)$ is the moduli space of semistable rank $r$ and degree $0$ Higgs vector bundles\footnote{On a Riemann surface, the second Chern character of any bundle vanishes automatically.}.
\item $\mathcal H^J_X(r)$ and $\mathcal H^K_X(r)$ are moduli spaces of flat ${\fontfamily{cmss}\selectfont\textbf{GL}}_r(\mathbb{C})$-connections on $X$, which are complex-analytically isomorphic to the moduli space of representations of $\pi_1 X$ into ${\fontfamily{cmss}\selectfont\textbf{GL}}_r(\mathbb{C})$.
\item The subset $\mathcal{H}^s_X(r)$ consisting of stable Higgs bundles (equivalently, irreducible representations of $\pi_1 X$) is a smooth, open, dense subvariety of $\mathcal H_X(r)$ with the structure of a Riemannian manifold.  
\item The metric on $\mathcal H^s_X(r)$ is complete, and $I$, $J$, and $K$ are covariantly constant and orthogonal with respect to the metric; in other words, $\mathcal{H}^s_X(r)$ is a complete hyperk\"ahler manifold.
\end{enumerate}
\end{theorem}

The moduli spaces in the theorem above are quotients of affine spaces of $\lambda$-connections by the notion of isomorphism for $\lambda$-connections established in \S\ref{section:Connections}. It is actually quite easy to see that $\mathcal H^I_X(r)$ and $\mathcal H^J_X(r)$ are not isomorphic as complex analytic spaces.  The space $\mathcal H^J_X(r)$ of flat connections is \emph{affine}, whereas $\mathcal H^I_X(r)$ contains a positive-dimensional \emph{projective} variety, namely the moduli variety of semistable vector bundles on $X$, embedded by $\mathcal E\mapsto(\mathcal E,0)$.  The singular locus is precisely $\mathcal H_X(r)\backslash\mathcal H^s_X(r)$, i.e., the locus of Higgs bundles that are semistable but not stable.

Fixing $r$, we may define a further sequence of noncompact varieties $\mathcal{H}_X(r,d)$, one for each $d\in\mathbb{Z}$, so that $\mathcal{H}^I_X(r,d)$ is the moduli space of semistable Higgs bundles of rank $r$ and fixed degree $d$, using the same stability condition as for the degree $0$ problem.  The spaces $\mathcal{H}^J_X(r,d)$ and $\mathcal{H}^K_X(r,d)$ are moduli spaces of representations of $\pi_1 X$ twisted by a root of unity, extending the moduli-theoretic nonabelian Hodge correspondence to arbitrary degree $d$ (see \cite{MR1990670,MR2453601,TH:13}, for instance).  The locus $\mathcal{H}^s_X(r,d)$ of stable Higgs bundles of degree $d$ is again smooth and complete hyperk\"ahler.  When $\gcd(r,d)=1$, semistable implies stable, and so $\mathcal{H}_X(r,d)=\mathcal{H}^s_X(r,d)$ is a smooth hyperk\"ahler manifold.  For now on, we will assume that $r$ and $d$ are such that $\gcd(r,d)=1$.  (In particular, $d\neq0$.)

Like all hyperk\"ahler spaces, $\mathcal H_X(r,d)$ comes with a sphere $S^2\cong\mathbb{P}^1$ of complex structures.  The complex structure $I$ lies at the pole $\zeta=0$ while $J$ and $K$ are equatorial, which we see by relating the sphere to the complex line of flat $\lambda$-connections via stereographic projection.  Since hyperk\"ahler implies Ricci flat, $\mathcal H^I_X(r,d)$ is a smooth Calabi-Yau manifold, albeit noncompact.

From now on, we will reserve the symbol $\mathcal M_X(r,d)$ for $\mathcal H^I_X(r,d)$.  Our emphasis on the moduli space of Higgs bundles over the other two modular interpretations will be justified by the presence of a natural $\mathbb{C}^\times$-action in \S\ref{section:CircleActionAndComplexVariationsOfHodgeStructure}.

\section{Properties of the moduli space of Higgs bundles on a curve}
\label{section:PropertiesOfModuliSpacesOfHiggsBundlesOnACurve}

Examining the varied and celebrated features of $\mathcal M_X(r,d)$ will expose for us a relationship between moduli of Higgs bundles and complex variations of Hodge structure.  As a warm up, we begin with the rank $1$ case.  We use the symbol $\omega_X$ for $\Omega^n(X)$, the top exterior power of the cotangent bundle.  Since $X$ is a curve from now on, Higgs fields for a bundle $\mathcal{E}$ are elements of $H^0(X,\sEnd\mathcal E\otimes\omega_X)$.  For economy, we sometimes use the juxtaposition $\mathcal E\mathcal F$ to mean the tensor product $\mathcal E\otimes\mathcal F$ of vector bundles $\mathcal E$ and $\mathcal F$. 

\subsection{The first example: the cotangent bundle of a Jacobian}
\label{subsection:CotangentBundleOfAJacobian}

Higgs line bundles over a curve garner little spotlight in the literature (although there are exceptions, e.g. \cite{MR2400111, MR2746119}).  While they are in some ways a toy model next to their higher-rank counterparts, Higgs line bundles betray many of the features for which Higgs bundles in general are famous.  For instance, the so-called Hitchin map on the moduli space of Higgs bundles is a proper map to an affine space.  Establishing this properness is nontrivial in general, but is on display for all to see in the rank 1 case.

Let $X$ be a nonsingular, connected curve of genus $g\geq1$.  A Higgs line bundle $(\mathcal L,\phi)$ consists of a holomorphic line bundle $\mathcal L$ and a section $\phi\in H^0(X,\sEnd(\mathcal{L})\otimes\omega_X)$, which is just a global holomorphic 1-form because $\sEnd \mathcal L \cong\mathcal O^\times$.  Holomorphic line bundles are stable vector bundles, and so the Higgs bundle $(\mathcal L,\phi)$ is stable in a rather trivial way for any $\mathcal L$ and $\phi$.  Therefore, the moduli space $\mathcal M_X(1,d)$ is a fibre bundle in which the base is the affine space $H^0(X,\omega_X)$ and the fibre is the $g$-dimensional complex torus $\mbox{Jac}^d(X)\cong T^g$.  Since $X$ has $g$ linearly-independent global holomorphic 1-forms, the total space of this fibration is $2g$-dimensional.  To see that this fibration is trivial, we appeal to a standard fact from deformation theory to conclude that the tangent space to the moduli space of line bundles of degree $d$ on $X$ at a line bundle $[\mathcal L]$ is\begin{eqnarray}T_{[\mathcal L]}(\mbox{Jac}^d(X))\;=\;H^1(X,\sEnd \mathcal L)\;\cong\;H^0(X,\sEnd \mathcal L\otimes\omega_X)^*,\nonumber\end{eqnarray}\noindent where the last isomorphism arises from Serre duality.  In other words, a cotangent vector to the point $[\mathcal L]$ is a Higgs field for $\mathcal L$.  Subsequently, the cotangent bundle to the moduli space of line bundles on $X$ is contained in the moduli space of Higgs line bundles on $X$.  It is easy to see that the containment is in both directions, and so we have\begin{eqnarray}\mathcal M_X(1,d) & = &  T^*\mbox{Jac}^d(X)\nonumber.\end{eqnarray}

Since the cotangent bundle of the complex torus $\mbox{Jac}^d(X)$ is trivial, so too is the moduli space of Higgs line bundles, considered as a vector bundle over $\mbox{Jac}^d(X)$.  It also follows immediately that $\mathcal M_X(1,d)$ is symplectic.

The group $\mathbb{C}^\times$ acts on $\mathcal M_X(1,d)$ by multiplication on $\phi$:\begin{eqnarray}(\mathcal L,\phi) & \stackrel{\mu}{\mapsto} & (\mathcal L,\mu\phi).\nonumber\end{eqnarray}\noindent  The only fixed points of the action are $(\mathcal L,0)$, where $0$ is the everywhere-vanishing 1-form on $X$.  In other words, the action fixes the fibre over zero in $\mathcal M_X(1,d)\cong T^g\times\mathbb{C}^g$.  It is significant that the fibre over zero consists exactly of the rank 1 complex variations of Hodge structure on $X$ (although the identification might seem merely coincidental at the moment).

Finally, the projection $\mathbf h:\mathcal M_X(1,d)\longrightarrow\mathbb{C}^g$ that forgets the line bundle but keeps the 1-form is obviously proper.

The following table indicates the relevant properties of $\mathcal M_X(1,d)$ and how they generalize to $\mathcal M_X(r,d)$.  In it, we use $\mathcal U_X(r,d)$ for the moduli space of semistable vector bundles of rank $r$ and degree $d$ on $X$, and the shorthand CVsHS for ``complex variations of Hodge structure''.

\vspace{10pt}
\footnotesize
\noindent
\begin{tabular*}{\textwidth}{@{\extracolsep{\fill}}ccc}
\toprule
& $\mathcal M_X(1,d)$ & $\mathcal M_X(r,d)$ \\ \midrule
Hitchin base & $H^0(X,\omega_X)$ & $\bigoplus_{i=1}^rH^0(X,\omega_X^{\otimes i})$ \\[2ex]
Fibres & Jac$^d(X)$ & generically, Jac$^{d'}(X_\rho)$ \\[2ex]
\begin{tabular}{c}Relation to \\ stable bundles\end{tabular} & $T^*\mathcal U_X(1,d)=\mathcal M_X(1,d)$ & $T^*\mathcal U_X(r,d)\subset\mathcal M_X(r,d)$ \\[3ex]
\begin{tabular}{c}Fixed points \\ of $\mathbb{C}^\times$-action \\ (CVsHS) \end{tabular} & $\mathbf h^{-1}(0)$ & subvarieties of $\mathbf h^{-1}(0)$ \\ \bottomrule
\end{tabular*}
\normalsize
\vspace{10pt}

In the remaining two sections, we will clarify the second column, including the number $d'$, the role of the additional curve $X_\rho$, and the way in which complex variations of Hodge structure enter the picture.

\subsection{The Hitchin map, spectral curves, and other features of $\mathcal M_X(r,d)$}
\label{subsection:HitchinMapSpectralCurvesAndOtherFeatures}

Now, we take $X$ to be a nonsingular, connected curve of genus $g\geq2$.  Consider the \emph{Hitchin map}\begin{eqnarray}\mathbf h\,:\,\mathcal M_X(r,d)\longrightarrow\mathbf B_r:=\displaystyle\bigoplus_{i=1}^rH^0(X,\omega_X^{\otimes i})\nonumber\end{eqnarray}\noindent defined by sending a class $[(\mathcal E,\phi)]$ to the $r$-tuple of coefficients of the characteristic polynomial, $\mbox{char}(\phi)$.  The affine space $\mathbf B_r$ is commonly referred to as the \emph{Hitchin base}.  While $\mathcal M_X(r,d)$ is noncompact for all $r$ and $d$, $\mathbf h$ is nevertheless proper.  This was established for $\mathcal M_X(2,d)$ by Hitchin in \cite{MR887284}.  For arbitrary rank $r$ on curves, the result was obtained by Nitsure in \cite{MR1085642}.  Finally, for arbitrary smooth projective varieties $X$, the result is due to Simpson in \cite{MR1320603}.

Each fibre of the Hitchin map has a modular interpretation, namely the moduli space of Higgs bundles with fixed $\mbox{char}(\phi)$. By the properness property, these moduli spaces are compact.

In order to better understand the geometry of the Hitchin map, it is advantageous to consider the \emph{spectral curve} associated to a Higgs bundle.  To define it, we use $S$ for the total space of $\omega_X$, $z$ for a coordinate on $S$, $\pi$ for the natural projection from $S$ onto $X$, and $\tau$ for the tautological section of the pullback $\pi^*\omega_X$.  

\begin{definition}  The \emph{spectral curve} of $\rho=(\rho_1,\dots,\rho_r)\in\mathbf B_r$ is the $1$-dimensional subvariety $X_\rho\subset S$ cut out by the equation\begin{eqnarray}\tau^r(z) & = & \rho_1(\pi(z))\tau^{r-1}(z)+\cdots+\rho_r(\pi(z)).\nonumber\end{eqnarray}\end{definition}

\noindent For a generic choice of $\rho$, $X_\rho$ is a nonsingular curve that is also a branched, $r:1$ cover of $X$. Now, take a line bundle $\mathcal L$ on such an $X_\rho$.  The direct image $(\pi|_{X_\rho})_*\mathcal L$ is a rank $r$ vector bundle $\mathcal E$ on $X$ (one only needs to check freeness at the branch points), and the pushforward of the multiplication map $\mathcal L\stackrel{\tau}{\longrightarrow}\pi^*\omega_X\otimes L$ is an $\mathcal O_X$-linear, $\omega_X$-valued endomorphism of $\mathcal E$ ---in other words, a Higgs field $\phi$ for the bundle $\mathcal E=(\pi|_{X_\rho})_*\mathcal L$.

Conversely, a Higgs bundle $(\mathcal E,\phi)$ determines a point $\rho\in\mathbf B_r$.  In turn, this point determines a spectral curve $X_\rho$, and the eigenspaces of $\phi$ form a line bundle $\mathcal L$ on $X_\rho$.   The branch points are the points $x\in X$ where $\phi_x$ has repeated eigenvalues.  

Rigorous developments of the correspondence between Higgs bundles on a curve and line bundles on a spectral curve can be found in \cite{MR885778,MR998478,MR1397273,MR1397059}.  For our purposes, we simply want a hint as to what kind of compact varieties the fibres of $\mathbf h$ are.  According to the spectral correspondence, the generic fibre $\mathbf h^{-1}(\rho)$ is the Jacobian of the spectral curve $X_\rho$.  The degree $d'$ of the line bundles in the Jacobian can be computed using the following argument of Hitchin \cite{MR1723384}.  Let $Y$ be a nonsingular curve and fix an ample line bundle $\mathcal J\rightarrow Y$ with $\deg\mathcal J=1$. For any vector bundle $\mathcal V\rightarrow Y$, there is a sufficiently large integer $n$ for which $H^0(Y,\mathcal V\otimes(\mathcal J^*)^{\otimes n})=0$.  Applying this to $\mathcal V^*\otimes\omega_Y$, it follows from Serre duality that\begin{eqnarray}H^1(Y,\mathcal V\otimes\mathcal J^{\otimes n})^*\;\cong\;H^0(Y,\mathcal V^*\otimes(\mathcal J^*)^{\otimes n}\otimes\omega_Y)\;=\;0.\nonumber\end{eqnarray}\noindent Thus, for a generic spectral curve $Y=X_\rho$ and a line bundle $\mathcal L$ on it, we have\begin{eqnarray}H^1(X_\rho,\mathcal L\otimes(\pi|_{X_\rho})^*\mathcal J^n)\;=\;0\nonumber\end{eqnarray}\noindent and\begin{eqnarray}H^1(X,\mathcal E\otimes\mathcal J^n)=0\nonumber\end{eqnarray}\noindent for $n$ sufficiently large, where $\mathcal E=(\pi|_{X_\rho})_*L$.   Because $\deg\mathcal J=1$, Riemann-Roch gives us\begin{eqnarray} h^0(X_\rho,\mathcal L\otimes(\pi|_{X_\rho})^*\mathcal J^n)=\deg\mathcal L+rn+(1-\tilde g)\nonumber\end{eqnarray}\noindent and\begin{eqnarray} h^0(X,(\pi|_{X_\rho})_*\mathcal L\otimes\mathcal J^n)=\deg\mathcal E+rn+r(1-g).\nonumber\end{eqnarray}\noindent It is a consequence of properties of the direct image functor that\begin{eqnarray}h^0(X_\rho,\mathcal L\otimes(\pi|_{X_\rho})^*\mathcal J^n) & = & h^0(X,(\pi|_{X_\rho})_*\mathcal L\otimes\mathcal J^n),\nonumber\end{eqnarray}\noindent and so\begin{eqnarray}d'\;=\;d-(1-\tilde g)+r(1-g),\nonumber\end{eqnarray}\noindent where $\tilde g$ is the genus of $X_\rho$, which can be computed via Riemann-Hurwitz using the branch divisor, or via the adjunction formula.

One fundamental piece of information we are missing is the dimension of $\mathcal M_X(r,d)$, and in particular how this dimension measures against that of $\mathbf B_r$.  For $\mathbf B_r$, note that Riemann-Roch and Serre duality tell us that\begin{eqnarray}h^0(X,\omega_X^{\otimes i})-h^0(X,\omega_X^{\otimes i-1}) & = & (2i-1)(g-1)\nonumber\end{eqnarray}\noindent for each $i\geq1$.   For $i\geq2$, $\omega_X^{\otimes i-1}$ has negative degree, and consequently is without global holomorphic sections.  The identity therefore reduces to\begin{eqnarray}h^0(X,\omega_X^{\otimes i}) & = & (2i-1)(g-1).\nonumber\end{eqnarray}\noindent  It follows that\begin{eqnarray}\dim\mathbf B_r & = & g+(g-1)\sum_{i=2}^r(2i-1)\nonumber\\ & = & g+2(g-1)\left(\frac{r(r+1)}{2}-1\right)-(r-1)\nonumber\\ & = & g+(r^2-1)(g-1)\nonumber\\ & = & 1+r^2(g-1).\nonumber\end{eqnarray}  

To compute the dimension of the total space of the Hitchin fibration, note that the tangent space to a point $[\mathcal E]\in\mathcal U_X(r,d)$ is\begin{eqnarray}T_{[\mathcal E]}(\mathcal U_X(r,d)) = H^1(X,\sEnd \mathcal E) \cong H^0(X,\sEnd\mathcal E\otimes\omega_X)^*,\nonumber\end{eqnarray}\noindent and so we have that the cotangent bundle to the moduli space of stable bundles is contained inside the moduli space of Higgs bundles, as was the case for Higgs line bundles.  The containment, however, is only in one direction this time.  Slope stability for Higgs bundles is a weaker condition that slope stability for bundles alone.  In particular, there are stable Higgs bundles $(\mathcal E,\phi)$ for which the underlying vector bundle $\mathcal E$ is unstable, and so the natural forgetful map $\mathcal M_X(r,d)\longrightarrow\mathcal U_X(r,d)$ is not globally defined in general.  However, the containment is open dense, and the symplectic form on $T^*\mathcal U_X(r,d)$ can be extended to one on $\mathcal M_X(r,d)$.  The complex structure $I$ on $\mathcal H_X(r,d)$ is also an extension of the one on $T^*\mathcal U_X(r,d)$.

Note that $H^0(X,\sEnd \mathcal E)\cong\mathbb{C}$ for a stable vector bundle $\mathcal E$ and that $\deg\sEnd \mathcal E = \deg(\mathcal E\otimes \mathcal E^*)=0$, and so the Riemann-Roch theorem tells us\begin{eqnarray}1-h^1(X,\sEnd \mathcal E) & = & r^2(1-g)\nonumber\end{eqnarray}\noindent and so\begin{eqnarray}\dim\mathcal U_X(r,d)\,=\,h^1(X,\sEnd\mathcal  E)\,=\,1+r^2(g-1),\nonumber\end{eqnarray}\noindent from which we conclude\begin{eqnarray}\dim\mathcal M_X(r,d)\,=\,2+2r^2(g-1).\nonumber\end{eqnarray}\noindent Notice that this is exactly twice the dimension of $\mathbf B_r$.

Juxtaposing all of this information, we observe that the complex symplectic manifold $\mathcal M_X(r,d)$ is a torus fibration over an affine half-dimensional base $\mathbf B_r$, and therefore is a Lagrangian fibration.  According to the theory of algebraic integrability, this makes $\mathcal M_X(r,d)$ into the total space of an algebraically completely integrable Hamiltonian system, referred to as \emph{the Hitchin system}. The first integrals, which pair-wise commute under the Poisson bracket coming from the symplectic form, are the components of the Hitchin map $\mathbf h$; that is, they are the invariant polynomials associated to $\phi$.  This beautiful facet of the theory of Higgs bundles was brought to light by Hitchin in \cite{MR885778}. Remarkably, many known integrable systems fit into the framework of the Hitchin system. The pushforward of a linear flow on the torus $\mbox{Jac}^d(X_\rho)$ gives rise to dynamics over $X$ in which the time evolution of $\phi$ is governed by a Lax pair equation, as explained in Hitchin's article in \cite{MR1723384}.

Because it is Calabi-Yau and because of its structure as a Lagrangian torus fibration, the moduli space of Higgs bundles on a curve is of interest from the point of view of Strominger-Yau-Zaslow mirror symmetry \cite{MR1429831}.  In particular, if we apply hyperk\"ahler rotation and shift to the moduli space of flat connections, the moduli space becomes a special Lagrangian fibration.  Far-reaching work on this aspect of moduli spaces of flat connections and Higgs bundles was carried out by Hausel and Thaddeus \cite{MR1990670}, and further developed in \cite{MR2306566,MR2354922,MR2648685,MR2957305} etc.  Mirror symmetry involving Higgs bundle moduli spaces is most interesting when we consider principal $G$-Higgs bundles for complex Lie groups other than the general linear group (for example, for $G={\fontfamily{cmss}\selectfont\textbf{SL}}_r(\mathbb{C}),\;{\fontfamily{cmss}\selectfont\textbf{PGL}}_r(\mathbb{C})$), as the appropriate mirror for the moduli space of $G$-Higgs bundles is the moduli space of $^{\fontfamily{cmss}\selectfont\textbf{L}}G$-Higgs bundles, where $^{\fontfamily{cmss}\selectfont\textbf{L}}G$ is the Langlands dual group.

\section{The circle action and complex variations of Hodge structure}
\label{section:CircleActionAndComplexVariationsOfHodgeStructure}

A major problem in the study of moduli spaces is to ascertain their topologies.  Hausel calls this the ``first approximation'' to a moduli problem \cite{TH:13}. For us, studying the topology means eliciting a sequence of numerical topological invariants, namely the Betti numbers.  In this section, we aim to show how the global topology of the moduli space of Higgs bundles is determined by a single fibre of the Hitchin map, and in particular by special Higgs bundles on this fibre that correspond to complex variations of Hodge structure.

When $\mathbb{C}^\times$ is allowed to act by multiplication on the twistor line, the fixed points are $\zeta=0,\infty$, where $0$ corresponds to $\mathcal M_X(r,d)$.  If $(\mathcal E,\phi)$ is a semistable Higgs bundle representing a class in $\mathcal M_X(r,d)$, then multiplication by $\mu$ sends this class to a class with representative Higgs bundle $\mu\cdot(\mathcal E,\phi)=(\mathcal E,\mu\cdot\phi)$, as defined in \S\ref{section:Connections}

For our purposes, we will restrict attention to the action of a compact subgroup, namely a circle $S^1=\{e^{i\theta}\in\mathbb{C}^\times:\,\theta\in\mathbb{R}\}$.  This allows us to bring in the tools of Morse-Bott theory, as was done by Hitchin in \cite{MR887284} and Gothen in \cite{MR1298999,PBG:95}.  The corresponding action is rotation of the Higgs field through $\theta$:\begin{eqnarray}\theta\cdot(\mathcal  E,\phi) & = & (\mathcal E,e^{i\theta}\phi).\nonumber\end{eqnarray}

Now we ask: which conditions must a Higgs bundle satisfy in order to be fixed under the $S^1$-action?  Of course, we are not asking for solutions to $\theta\cdot(\mathcal E,\phi)=(\mathcal E,\phi)$ in the strict sense.  Strictly speaking, the only solutions are pairs $(\mathcal E,0)$.  In other words, the solution set is exactly the moduli space of stable rank $r$, degree $d$ bundles on $X$.  In truth, there are other solutions, because we are working at the level of moduli, meaning that we find solutions of $e^{i\theta}\cdot(\mathcal E,\phi)=(\mathcal E,\phi)$ up to the notion of isomorphism of $\lambda$-connections supplied in \S\ref{section:Connections}.  More precisely:

\begin{definition} A point $[(\mathcal E,\phi)]\in\mathcal M_X(r,d)$ is a \emph{fixed point} of the $S^1$-action on $\mathcal M_X(r,d)$ if and only if there exists, for each $\theta\in\mathbb{R}$, an automorphism $\psi_\theta$ of $\mathcal E$ so that $\psi_\theta\phi\psi^{-1}_\theta=e^{i\theta}\phi$.\end{definition}

To achieve a complete characterization of the fixed points, we follow \cite{MR887284,MR1159261,PBG:95}.  Suppose that $(\mathcal E,\phi)$ is a fixed point.  Denote by $\Psi$ the infinitesimal generator of the associated one-parameter family of automorphisms; that is,\begin{eqnarray}\Psi & := & \left.\frac{d(\psi_\theta)}{d\theta}\right|_{\theta=0},\nonumber\end{eqnarray}\noindent and $\Psi$ is an endomorphism of $\mathcal E$.   Differentiating both sides of the fixed point equation gives us\begin{eqnarray}\frac{d(\psi_\theta)}{d\theta}\phi\psi^{-1}_\theta-\psi_\theta\phi\psi_\theta^{-1}\frac{d(\psi_\theta)}{d\theta}\psi_\theta^{-1} & = & ie^{i\theta}\phi\nonumber,\end{eqnarray}\noindent and evaluating both sides of this at $\theta=0$ yields\begin{eqnarray}\Psi\phi 1_\mathcal E-1_\mathcal E\phi 1_\mathcal E\Psi1_\mathcal E & = & i\phi\nonumber,\end{eqnarray}\noindent which is just a commutator relation, specifically\begin{eqnarray}\left[\Psi,\phi\right] & = & i\phi\label{CommRel}.\end{eqnarray}\noindent

Recall that the Koszul-Malgrange integrability theorem equates the holomorphic structure $\mathcal E$ on $E$ to a $\mathbb{C}$-linear operator\begin{eqnarray}\overline\partial_E\,:E\longrightarrow E\otimes_{\mathcal A_X^0}\mathcal A_X^{0,1}\nonumber\end{eqnarray}\noindent satisfying the Leibniz property and the integrability condition $\overline\partial_E^2=0$.  If $(\mathcal E,\Phi)$ is a fixed point, then we must also have\begin{eqnarray}\psi_\theta^{-1}\overline\partial_E\psi_\theta & = & \overline\partial_E\label{FixedPartial}\end{eqnarray}\noindent for all $\theta\in\mathbb{R}$.  The automorphism $\psi_\theta$ is not $1_\mathcal E$ whenever $\theta$ is not an integer multiple of $2\pi$.  Consequently, in order for equation \eqref{FixedPartial} to hold, it must be that the holomorphic structure on $\mathcal E$ is reducible.  Specifically, $\mathcal E$ must decompose globally into a direct sum of eigenspace bundles of the endomorphism $\Psi$.  Call these eigenspace bundles $\Lambda_1,\dots,\Lambda_n$.  The eigenvalues of $\Psi$ are global holomorphic sections of $\omega_X$, and we call them $\lambda_1,\dots,\lambda_n$, respectively.  Note that the only condition on $n$ is that $n\leq r$; it is possible to have fixed points for $n=1$, in which case $\Lambda_1=\mathcal E$, although $\phi=0$ in this case.  At the other extreme, $n=r$ implies that each $\Lambda_j$ is a line bundle on $X$.

Pick a bundle $\Lambda_j$ and let both sides of the commutator relation \eqref{CommRel} act on it.  The result is\begin{eqnarray} \Psi\phi(\Lambda_j) - \lambda_j\phi(\Lambda_j) & = & i\phi(\Lambda_j),\nonumber\end{eqnarray}\noindent which rearranges into\begin{eqnarray}\Psi(\phi\Lambda_j) & = & (\lambda_j+i)(\phi\Lambda_j).\nonumber\end{eqnarray}\noindent In other words, $\phi(\Lambda_j)$ is contained in an eigenspace bundle of $\Psi$ ---specifically, the one corresponding to eigenvalue $\lambda_j+i$.

We therefore have the following characterization: a rank $r$ Higgs bundle $(\mathcal E,\phi)$ is a fixed point of the $S^1$-action if and only if there exist natural numbers $1=m_0\leq m_1\leq\cdots\leq m_k=n\leq r$ such that $E=\bigoplus_{j=1}^n\Lambda_j$ and\begin{eqnarray}\Lambda_1\stackrel{\phi}{\longrightarrow}\cdots\stackrel{\phi}{\longrightarrow}\Lambda_{m_1}\otimes(\omega_X)^{\otimes(m_1-1)}\stackrel{\phi}{\longrightarrow} &  0 & \nonumber\\ & \vdots & \nonumber\\\Lambda_{m_{k}+1}\stackrel{\phi}{\longrightarrow}\cdots\stackrel{\phi}{\longrightarrow}\Lambda_{n}\otimes(\omega_X)^{\otimes(n-(m_{k-1}+1))}\stackrel{\phi}{\longrightarrow} & 0 &,\nonumber\end{eqnarray}\noindent where the only zero maps are the ones at the ends of the sequences, and the eigenvalue of $\Lambda_{m_j+p}$ is $\lambda_{m_j+p}=\lambda_{m_j}+pi$ for each $p=0,1,\dots,m_{j+1}-m_j-1$.

In other words, the action of the Higgs field $\phi$ is described by several sequences of bundles.  Each of these sequences gives rise to a sub-bundle $\bigoplus_{m_{j-1}+1}^{m_j}\Lambda_j$ of $\mathcal E$ that is invariant under $\phi$ (up to twisting by powers of $\omega_X$).  It is easy to show that, for numerical reasons, a stable Higgs bundle $(\mathcal E,\phi)$ cannot be a direct sum of two or more $\phi$-invariant subbundles.  Consequently, a rank $r$ stable Higgs bundle $(\mathcal E,\phi)$ is a fixed point of the $S^1$-action if and only if there exists a natural number $n\leq r$ such that $\mathcal E=\bigoplus_{j=1}^n\Lambda_j$ and\begin{eqnarray}\Lambda_1\stackrel{\phi}{\longrightarrow}\cdots\stackrel{\phi}{\longrightarrow}\Lambda_{n}\otimes(\omega_X)^{\otimes(n-1)}\stackrel{\phi}{\longrightarrow}0,\nonumber\end{eqnarray}\noindent where the only zero map is the one at the end of the sequence.   (We have not repeated the statement regarding the eigenvalues because it is unnecessary for characterizing the fixed points.)

From this discussion, we highlight two important facts:\begin{enumerate}\item The fixed points in $\mathcal M_X(r,d)$ of $S^1$ are precisely the Higgs bundles that are complex variations of Hodge structure.  This connection was observed by Simpson \cite{MR944577,MR1159261}.\item  A fixed point is necessarily nilpotent, implying that each and every fixed point is a point in the Hitchin fibre $\mathbf h^{-1}(0)$.\end{enumerate}

This fibre is referred to as the \emph{nilpotent cone}, and we shall use the notation $\mbox{Nilp}_X(r,d):=\mathbf h^{-1}(0)\subset\mathcal M_X(r,d)$ from now on.

To establish a link between the data in $\mbox{Nilp}_X(r,d)$ and the topology of $\mathcal M_X(r,d)$, we point out the presence of a Morse--Bott function\begin{eqnarray}f\,:\,\mathcal M_X(r,d)\longrightarrow\mathbb{R}_{\geq0}\nonumber\end{eqnarray}\noindent whose value is a scalar multiple of the norm squared of $\phi$, using the norm coming from the K\"ahler metric associated to the complex structure $I$.  This function is a proper moment map for the $S^1$-action and satisfies the conditions for being a perfect, nondegenerate Morse-Bott function, as established in \S7 of Hitchin's paper \cite{MR887284}.  

We will not have to study the function directly, at least for our purposes.  We simply note the following substantial consequence of its existence: the downward gradient flow of $f$ coincides with $\mbox{Nilp}_X(r,d)$, which is a deformation retract of $\mathcal M_X(r,d)$ (as in \S4 of \cite{TH:98}).

If $(\mathcal E,\phi)$ is a fixed point, then let $(r_1,\dots,r_n)$ be the $n$-tuple of ranks of the (ordered) $\Lambda_j$ bundles.  The vector $(r_1,\dots,r_n)$ is referred to as the \emph{type} of the fixed point.  For example, the moduli space $\mathcal U_X(r,d)$ of stable bundles of rank $r$ and degree $d$ is the locus of type $(r)$ fixed points in $\mbox{Nilp}_X(r,d)$.

The primary tool for understanding the topology of $\mathcal M_X(r,d)$ is Hitchin's Morse-Bott localization procedure, which appeared in its earliest form in \S7 of \cite{MR887284}. 

\begin{theorem} \emph{(Hitchin, \cite{MR887284})}  The Betti numbers for the ordinary rational cohomology of $\mathcal M_X(r,d)$ are generated by \begin{eqnarray}\mathcal P_y(r,d) & = & \sum y^{\beta(\mathcal N)}\mathcal P_y(\mathcal N),\label{Local1}\end{eqnarray}\noindent where the sum is taken over connected components $\mathcal N$ of the fixed point set of $S^1\circlearrowright\mathcal M_X(r,d)$, and where $\beta(\mathcal N)$ is the number of negative eigenvalues of $\emph{\mbox{Hess}}(f)$ at any fixed point in $\mathcal N$.\end{theorem}

This identity asserts that the topological data of $\mathcal M_X(r,d)$ are controlled by the subvarieties $\mathcal N\subset\mbox{Nilp}_X(r,d)$ which are moduli spaces of complex variations of Hodge structure, by the correspondence with fixed points of the $S^1$-action.

The number $\beta(\mathcal N)$ in the identity \eqref{Local1} is the \emph{Morse index} of standard Morse theory.  It counts the degrees of freedom for the downward flow out of a component $\mathcal N$.  The main difference between standard Morse theory for compact spaces and the Morse-Bott theory used for the noncompact space $\mathcal M_X(r,d)$ is that the fixed point set is not necessarily discrete.  Indeed, the terminal component of the flow, which is also the component with index $0$, is $\mathcal U_X(r,d)$, which is a connected subvariety of $\mbox{Nilp}_X(r,d)$ of positive dimension.  This is generally true for all components, and the implication is that the calculation of the right-hand side of \eqref{Local1} is highly nontrivial, even for low rank.

While determining $\mathcal P_y(\mathcal N)$ is a major obstacle to calculating $\mathcal P_y(r,d)$, the Morse index submits to direct calculation without too much trouble.\footnote{The Morse index is essentially deformation theoretic.  The Higgs field in a complex variation of Hodge structure shifts the sequence of $\Lambda_j$ bundles by one step when it acts, i.e. $\phi:\Lambda_j\rightarrow\Lambda_{j+1}\otimes\omega_X$.  Roughly speaking, the downward flow arises algebraically by considering shifts of weight at least 2.}  There are various ways to compute $\beta$.  Hitchin does so in \cite{MR887284} by exploiting the symplectic structure.  In \cite{PBG:95}, Gothen uses only the K\"ahler form and deformation theory.  Using arguments in the style of Gothen, it can be shown that $\beta(\mathcal E,\phi)$ for a fixed point $(\mathcal E,\phi)$ is \begin{eqnarray} (4g-4)\delta^n_2\sum_{i=1}^{n-2}\sum_{j=i+2}^nr_ir_j-2\delta^n_1\sum_{i=1}^{n-1}(\deg(\Lambda_i^*\otimes\Lambda_{i+1})+(1-g)r_ir_{i+1})\label{IndexFormula}\end{eqnarray}\noindent 
where $\delta^n_j=1$ if $n>j$ and $0$ otherwise.\footnote{We were not able to find an appearance of this exact formula in the literature.  In \S7 of \cite{MR887284}, the rank 2 case of this formula is derived.  It is shown in \cite{MR3158239} how to extract a formula for all ranks from Gothen's calculation of the Morse index, in the case of co-called ``co-Higgs bundles'' on the projective line.  Superficial modifications to the procedure will produce the formula presented above.  See also \cite{MR2308696} for a similar formula in the parabolic Higgs setting.}

To appease the reader concerned as to why $\beta$ is constant on a component $\mathcal N$, we point out a result, credited to Simpson and written down as Lemma 9.2 in \cite{MR1990670}, guaranteeing that the type $(r_1,\dots,r_n)$ and the degrees of the $\Lambda_j$ bundles are constant on a connected component.  Note that\begin{eqnarray}\deg(\Lambda_i^*\otimes\Lambda_{i+1})=-r_{i+1}\deg \Lambda_i+r_i\deg \Lambda_{i+1},\nonumber\end{eqnarray}\noindent and so the right-hand side of \eqref{IndexFormula} depends only on the ranks and degrees of the bundles in the representation of $(\mathcal E,\phi)$ as a variation of Hodge structure.  It follows that the number $\beta(\mathcal N)$ is well-defined.

\section{Example calculations}
\label{section:ExampleCalculations}

\subsection{$\mathcal P_y(1,d)$}
\label{subsection:P1d}

The entire nilpotent cone of $\mathcal M_X(1,d)$ is fixed by the circle action: it is the set $\{(\mathcal L, 0) :\,\mathcal{L}\in\mbox{Jac}^d(X)\}$, which is precisely the set of fixed points of type $(1)$, that is, the moduli space of line bundles of fixed degree.  Hence,\begin{eqnarray}\mathcal P_y(\mathcal M_X(1,d))\;=\; y^0\mathcal P_y(\mbox{Jac}^d(X))\;=\;(1+y)^{2g},\nonumber\end{eqnarray}\noindent agreeing with the explicit description of the moduli space as a fibre bundle $\mathcal M_X(1,d)\cong T^g\times\mathbb{C}^g$.

\subsection{$\mathcal P_y(2,1)$}
\label{subsection:P21}

We turn to the task of computing Betti numbers of $\mathcal M_X(2,1)$, where $X$ is a nonsingular, connected projective algebraic curve of genus $g\geq2$.  The generating function was originally isolated by Hitchin in \S7 of \cite{MR887284}.  The outline here is essentially a reproduction of Hitchin's calculations, although we leave out certain important but cumbersome details.  In particular, we will only present the Betti numbers themselves in the case of $g=2$.

We also need to make clear that Hitchin's calculation is for the ${\fontfamily{cmss}\selectfont\textbf{SL}}_2(\mathbb{C})$-Higgs bundle moduli space, obtained by fixing the determinant $\wedge^2\mathcal E=\mathcal W$, where $\mathcal W$ is some line bundle of degree 1 on $X$.  We denote by $\mathcal M_X(2,\mathcal W)$ the moduli space with this restriction.  On this space, there is an action of the group $\mbox{Jac}^0_2(X)$ of line bundles on $X$ whose square is trivial, by the tensor product with $\mathcal E$ in each $(\mathcal E,\phi)$.  The cohomology of $\mathcal M_X(2,\mathcal W)$ has a part that is \emph{invariant} under this action and a part that is \emph{variant}, to use the terminology of \cite{MR1990670}.  There,  Hausel and Thaddeus define the ${\fontfamily{cmss}\selectfont\textbf{PGL}}_2(\mathbb{C})$ moduli space $\widetilde{\mathcal M}_X(2,\mathcal W)$ as the quotient of $\mathcal M_X(2,\mathcal W)$ by this action.  (This is a geometric quotient by a group scheme.)  In \cite{MR2453601}, it is shown that the generating function for the Betti numbers of the ${\fontfamily{cmss}\selectfont\textbf{GL}}_r(\mathbb{C})$ moduli space factors as\begin{equation}\mathcal P_y(r,1) = (1+y)^{2g}P_y(\widetilde{\mathcal M}_X(r,\mathcal W)).\label{DecompPoincareGL}\end{equation}

The final result of Hitchin's calculation differs slightly from what appears below because his keeps track of both the invariant and the variant cohomology, whereas we need only the invariant cohomology for ours, by \eqref{DecompPoincareGL}. Nevertheless, we can extract the ranks of the invariant pieces of the cohomology from Hitchin's calculation.

By the arguments in \S\ref{section:CircleActionAndComplexVariationsOfHodgeStructure}, the elements of the fixed point set are Higgs bundles of one of two types.  The first type is $(\mathcal E,\phi)$ with $\mathcal E=\mathcal J\oplus \mathcal J^*\otimes\mathcal W$ for some line bundle $\mathcal J$ on $X$, and $\phi:\mathcal J\rightarrow\mathcal J^*\otimes\mathcal W\otimes\omega_X$.  The data of such a fixed point is therefore a pair $(\mathcal J,\phi)$, where $\mathcal J$ is a line bundle and $\phi$ is a section of the line bundle $(\mathcal J^*)^{2}\mathcal W\omega_X$.  These are the fixed points of type $(1,1)$.  The other type is $(\mathcal E,0)$, where $\mathcal E$ is any rank 2 stable bundle with $\wedge^2\mathcal E=\mathcal W$.  These are the fixed points of type $(2)$.

The points of type $(2)$ form a copy of the moduli space $\mathcal{U}_X(2,\mathcal W)$ of stable rank 2 bundles with the fixed degree 1 determinant $\mathcal  W$.  This copy sits inside the nilpotent cone of $\mathcal M_X(2,\mathcal W)$ and is the limit of the downward flow of the Morse-Bott function.  As such, the Morse index for this component is $0$, which agrees with formula \eqref{IndexFormula}.

For type $(1,1)$ we ask: which line bundles $\mathcal J$ are permissible?  Let $m=\deg(\mathcal J)$.  Since $\mathcal J^*\mathcal W$ is the kernel of $\phi$, we must have $-m+1<\frac{1}{2}$, by stability.  If $\phi$ were identically zero, then $\mathcal J$ would be an invariant sub-line bundle with slope exceeding $\frac{1}{2}$, and so the resulting Higgs bundle would be unstable.  Subsequently, slope stability necessitates that $\phi\neq0$, which in turn necessitates that the degree of $(\mathcal J^*)^{2}\mathcal W\omega_X$ is nonnegative: $-2m+1+2g-2\geq0$.  Combining this information, we have the condition\begin{eqnarray}\displaystyle1\,\leq \,m\,\leq\, g-1.\label{Interval}\end{eqnarray}

\begin{theorem} \emph{(Hitchin, \cite{MR887284})} Let $\mathcal N^{1,1}$ be the set of type $(1,1)$ fixed points in $\mathcal M_X(2,\mathcal W)$.  Each connected component of $\mathcal N^{1,1}$ is a degree $2^{2g}$ covering of ${\fontfamily{cmss}\selectfont\textbf{S}}^{-2m+2g-1}(X)$, the symmetric product of $X$ with itself $-2m+2g-1$ times, for some integer $m$ in the interval \eqref{Interval}.\end{theorem}

\begin{proof} Pick a line bundle $\mathcal J\in\mbox{Jac}^m(X)$, where $m$ satisfies \eqref{Interval}, and such that $H^0(X,(\mathcal J^*)^2\mathcal W\omega_X)\backslash\{0\}$ is nonempty.  The vanishing locus of each $\phi\in H^0(X,(\mathcal J^*)^2\mathcal W\omega_X)\backslash\{0\}$ is an effective divisor $D\in{\fontfamily{cmss}\selectfont\textbf{S}}^{-2m+2g-1}(X)$.  Conversely, the data of $D$ determines a line bundle $\mathcal Q$ of degree $2m$ and a holomorphic section $[\phi]\in\mathbb{P}(H^0(X,\mathcal Q^*\mathcal W\omega_X)\backslash\{0\})$ vanishing on $D$.  To recover $\mathcal J$, we take a holomorphic square root of $\mathcal{Q}$ in $\mbox{Jac}^m(X)$.  Such roots exist because $\mathcal{Q}$ has even degree.  There are $2^{2g}$ such square roots.

Although $D$ recovers $[\phi]$ only up to projective equivalence, the line bundle $\mathcal J$ has $\mathbb{C}^\times$-many automorphisms, which identify elements lying on the same line in $H^0(X,(\mathcal J^*)^2\mathcal W\omega_X)\backslash\{0\}$.  Consequently, two pairs $(\mathcal J,\phi)$ and $(\mathcal J,c\phi)$ which differ only by some $c\in\mathbb{C}^\times$ represent the same fixed point in the moduli space.

Finally, we recall from above that, since $\mathcal J$ is an eigenspace for the generator $\Psi$, its degree is constant on connected components of the fixed point set of the $S^1$ action.\qed\end{proof}

Let us restrict now to the case where the genus is $2$.  The Betti numbers of the moduli space of rank 2 vector bundles with odd degree and (any) fixed determinant are encoded by\begin{eqnarray}1+y^2+4y^3+y^4+y^6,\nonumber\end{eqnarray}\noindent according to the calculation by Atiyah and Bott \cite{MR702806}.  The moduli space of bundles with fixed determinant contributes nothing to the variant cohomology of $\mathcal M(2,\mathcal W)$, as per \cite{MR0364254}, and so we append this polynomial to our overall calculation without further adjustment.

As for the type $(1,1)$ fixed points, there is only one permissible degree for the line bundle $\mathcal J$, which is $m=1$.  The result is that $\mathcal N^{1,1}$ is connected and is a $16$-fold covering of $X$ itself (since $-2m+2g-1=1$).  However, we wish to find the Betti numbers of $\widetilde{\mathcal M}_X(2,\mathcal W)$ rather than those of $\mathcal M_X(2,\mathcal W)$.  The action of $\mbox{Jac}^0_2(X)$ is transitive on the fibres of the covering of $X$, and so the corresponding fixed point set of the $S^1$ action in $\widetilde{\mathcal M}_X(2,\mathcal W)$ is just a copy of $X$ itself\footnote{The $\mathbb{C}^\times$-action commutes with the action of $\mbox{Jac}^0_2(X)$ \cite{MR1990670}.}. Therefore, the Betti numbers of the type $(1,1)$ locus are those of $X$ itself: $b_0=b_2=1$, $b_1=4$, and $b_k=0$ for $k>2$.

Finally, by \eqref{IndexFormula} and \eqref{DecompPoincareGL}, the Betti numbers of $\mathcal M_X(2,1)$ are generated in the case of $g=2$ by\footnote{Hitchin's $g=2$ generating function in \cite{MR887284} for the ${\fontfamily{cmss}\selectfont\textbf{SL}}_2(\mathbb{C})$-Higgs bundle moduli space is $1+y^2+4y^3+2y^4+34y^5+2y^6$, which is the same as the second factor of our presentation of $\mathcal P_y(2,1)$ save for a difference in the coefficient of $y^5$.}\begin{eqnarray}\mathcal P_y(2,1) & = & (1+y)^{4}(\mathcal P_y(\mathcal U_X(2,P))+y^4\mathcal P_y(\mathcal N^{1,1}))\nonumber\\ & = & (1+y)^4(1+y^2+4y^3+y^4+y^6+y^4(1+4y+y^2))\nonumber\\ & = & (1+y)^4(1+y^2+4y^3+2y^4+4y^5+2y^6)\nonumber.\end{eqnarray}

Readers interested not only in the Betti numbers but also in the actual structure of the cohomology ring of $\mathcal M_X(2,1)$ should consult \cite{MR1949162,MR2044052,MR1887889}, where the generators and relations are worked out explicitly.

Betti numbers for the rank $3$ case were calculated by Gothen in \cite{MR1298999,PBG:95}. The underlying technique is the same, but a major complication in that case comes from the existence of fixed points of types $(2,1)$ and $(1,2)$.  For these types, the calculation relies in a crucial way upon results of Thaddeus on moduli spaces of vector bundles with sections \cite{MR1273268}. Fixed points of type $(2,1)$ and $(1,2)$ are special cases of what are called \emph{holomorphic triples}, studied extensively in the case of compact Riemann surfaces of genus $g\geq2$ in \cite{MR2030379}.  The rank $4$ Betti numbers were computed recently in \cite{GHS:11}, using a motivic version of Hitchin's method.

\subsection{Degree independence of Betti numbers}
\label{subsection:DegreeIndependenceOfBettiNumbers}

Finally, we remark that Betti numbers of $\mathcal M_X(r,d)$ are independent of $d$.  This was shown by Hausel and Rodriguez-Villegas in \cite{MR2453601}.  They proved first that the Betti numbers of the moduli space of twisted representations of $\pi_1 X$ are independent of the root of unity used to define the twist.  It follows from the diffeomorphism afforded by nonabelian Hodge theory that $\mathcal M_X(r,d)$ has Betti numbers independent of $d$.

This is significant because there is no reason \emph{a priori} to assume that the subvarieties of $\mbox{Nilp}_X(r,d)$ and $\mbox{Nilp}_X(r,d')$ consisting of complex variations of Hodge structure have, for instance, the same number of connected components.

\bibliographystyle{spmpsci}
\bibliography{IntroNAHT}

\end{document}